\documentclass[11pt,letterpaper]{amsart}

\usepackage[utf8]{inputenc}
\usepackage{amssymb, amscd, amsmath, epsfig, mathtools}
\usepackage{amsthm}
\usepackage{enumerate}
\usepackage{xcolor}
\usepackage{scalerel}
\usepackage{soul}
\usepackage{tikz-cd}
\usepackage{esint}
\usepackage{cancel}
\usepackage{hyperref}
\usepackage{slashed}
\usepackage{url}

\usepackage[margin=1.0in]{geometry}

\newtheorem{theorem}{Theorem}[section]
\newtheorem*{theorem*}{Theorem}

\newtheorem{corollary}{Corollary}[section]
\newtheorem*{corollary*}{Corollary}
\newtheorem{lemma}{Lemma}[section]
\newtheorem{proposition}{Proposition}[section]
\theoremstyle{definition}
\newtheorem{remark}{Remark}[section]

\newcommand{\R}{\mathbb R}

\newcommand{\calC}{\mathcal C}

\newcommand{\calL}{\mathcal L}

\begin{document}

\title{Scalar and Mean Curvature Comparison on Compact Cylinder}
\author{Jie Xu}
\address{
Department of Technology, Operations and Statistics, New York University, New York, NY, USA}
\email{jx479@nyu.edu}

\begin{abstract}
Let $ X $ be a closed, oriented Riemannian manifold. Denote by $ (M = X \times I = X \times [0, 1], \partial M = X \times \lbrace 0 \rbrace \cup X \times \lbrace 1 \rbrace, g) $ a compact cylinder with smooth boundary, $ \dim M \geqslant 3 $. In this article, we address the following question:  If $ g $ is a Riemannian metric having (i) positive scalar curvature (PSC metric) on $ M $ and nonnegative mean curvature on $ \partial M $; and (ii) $ \chi(X) = 0 $ and $ \dim X \geq 5 $, then there exists a PSC metric on $ X $. This generalizes a result of Gromov and Lawson \cite{GL} for $ X = \mathbb{T}^{n} $. Our result also implies an infinite band width consequence complementary to Gromov's band width conjecture, which roughly says no additional obstruction or flexibility for the existence of PSC on $ X $ is gained by passing to $ X \times [0, 1] $. Our result relies on geometric PDE analysis with respect to a new conformal invariance---the $ g $-angle condition between normal vector field along $ \partial M $ and the global vector field tangential to the $ I $-direction. 
\end{abstract}

\maketitle

\section{Introduction}
Let $ X $ be a closed, oriented manifold with $ n - 1: = \dim X \geqslant 2 $. Let $ I = [0, 1] $ be the standard unit interval with global coordinate $ \xi $. The relation between the geometry of $ X $ and the geometry of the compact cylinder $ X \times I $ is an active research area.

In Gromov's seminal work \cite{GROMOV}, \cite{GROMOV2}, the study of Riemannian bands on compact cylinders with lower scalar curvature bounds has been an active field for recent years. On compact cylinders $ X \times I $, the Riemannian band width is defined to be the distance between $ X \times \lbrace 0 \rbrace $ and $ X \times \lbrace 1 \rbrace $ with respect to $ g $. Gromov conjectured an upper bound of Riemannian bands for a class of manifolds $ (X \times I, g) $, where $ X $ admits no Riemannian metric with positive scalar curvature (PSC), but $ X \times I $ admits a metric $ g $ whose associated scalar curvature is uniformly positive, i.e. there exists a positive constant $ \kappa_{0} > 0 $ such that the scalar curvature is bounded below by $ \kappa_{0} $. The model case is $ X = \mathbb{T}^{n}, n \geqslant 1 $, see \cite{GROMOV2}. Recently, this conjecture has been proved for a wide class of compact cylinders with some extra hypotheses, especially when $ X $ is a spin manifold, see e.g. \cite{CZ}, \cite{Rade}, \cite{Zeidler}. 

Gromov's conjecture on Riemannian bands gives a metric geometry information on compact cylinders. There is also a Riemannian geometry information in terms of mean curvature on compact cylinders. Gromov and Lawson \cite{GL} showed that if $ X = \mathbb{T}^{n}, n \geqslant 1 $, and $ X \times I = \mathbb{T}^{n} \times I $ admits a PSC metric $ g $, then the mean curvature $ h_{g} $ must be negative somewhere on $ X \times \lbrace 0 \rbrace \cup X \times \lbrace 1 \rbrace $. Very recently, R\"ade \cite{Rade} proved this for general compact cylinders up to $ \dim X \leqslant 7 $, except some cases when $ \dim X = 4 $. We point out that in these dimensions, a finite upper bound of Riemannian band implies that the mean curvature is negative somewhere, but the opposite direction is not true in general.

In this article, we generalize the result of Gromov and Lawson \cite{GL} on $ \mathbb{T}^{n} $ to all compact, oriented manifolds $ X $ with $ \dim X \geqslant 5 $ and $ \chi(X) = 0 $. The topological constraints are associated with a $ g $-angle condition defined in (\ref{intro:eqn1}) and a geometric PDE analysis with the introduction of an auxiliary space developed in \cite{RX2}, see (\ref{diagram}) below. 
\begin{theorem}\label{intro:thm0}
Let $ X $ be a connected, closed, oriented manifold, $ \dim X \geqslant 5 $, with $\chi(X) = 0$. 
Then $ X \times I $ admits a metric $ g $ with positive scalar curvature and nonnegative mean curvature if and only if $ X $ admits a PSC metric.
\end{theorem}
\noindent Immediately, we show that Theorem \ref{intro:thm0} holds in every odd dimension of $ X $ at least $ 5 $. The contrapositive statement of the theorem generalizes the result of Gromov and Lawson \cite{GL} on compact cylinders $ X \times I $ from the torus case $ X = \mathbb{T}^{n} $ to all closed, oriented manifolds $ X $ with $ \dim X \geqslant 5 $ and $ \chi(X)  = 0 $. Clearly $ \chi(\mathbb{T}^{n}) = 0, \forall n \geq 1 $.
\begin{corollary}
    Let $ X $ be a closed, oriented manifold with $ \dim X \geqslant 5 $ and $ \chi(X) = 0 $. If $ X $ admits no PSC metric, meanwhile $ X \times I $ admits a PSC metric $ g $, then the mean curvature $ h_{g} $ must be negative somewhere on $ X \times \lbrace 0 \rbrace \cup X \times \lbrace 1 \rbrace $.
\end{corollary}
The corollary also generalizes the result of R\"ade \cite{Rade} by removing the dimension hypothesis.
\medskip
Theorem \ref{intro:thm0} relies on (i) a geometric result with a conformally invariant $ g $-angle condition when $ \dim X \geq 2 $, and (ii) the connection between a nowhere vanishing global vector field on $ \Gamma(T(X \times I)) $ other than $ \partial_{\xi} $ and the $ g $-angle condition. 

For notations throughout this article, let $ (M = X \times I, \partial M = X \times \lbrace 0 \rbrace \cup X \times \lbrace 1 \rbrace, g) $ be a compact cylinder with $ n : = \dim M \geqslant 3 $. We may also assume that $ X $ is connected throughout this article, since otherwise our analysis reduces to each connected component of $ X $. Let $ R_{g} $ be the scalar curvature with respect to $ g $ on $ M $, and let $ R_{\imath^{*}g} $ be the scalar curvature for the induced metric $ R_{\imath^{*}g} $ on the closed manifold $ X \cong X \times \lbrace 0 \rbrace $ with respect to the natural inclusion $ \imath :X \times \lbrace 0 \rbrace \rightarrow M $. Let $ \nu_{g} $ be the positively oriented unit normal vector field along $ \partial M $, and inwarding along $ X \times \lbrace 0 \rbrace $. Let $ h_{g} $ be the mean curvature with respect to $ g $ on $ \partial M $. We denote the conformal class $ [g] $ of $ g $ on $ M $ by
\begin{equation*}
    [g] = \lbrace e^{2\phi}g : \phi \in \calC^{\infty}(M) \rbrace.
\end{equation*}
We sometimes will label $ I $ by $ I_{\xi} $ for clarity.

On $ M = X \times I_{\xi} $, the $ g $-angle between $ \nu_{g} $ and the canonical vector field $ \partial_{\xi} $ along $ X \times \{ 0 \} $, which is defined to be
\begin{equation}\label{intro:eqn0}
    \angle_{g}\left(\nu_{g},\partial_{\xi} \right)  = \cos^{-1}\left(\frac{g(\nu_{g}, \partial_{\xi})}{\left(g(\nu_{g}, \nu_{g} \right))^{\frac{1}{2}}\left(g(\partial_{\xi} , \partial_{\xi}  \right))^{\frac{1}{2}}} \right) = \cos^{-1}\left(\frac{g(\nu_{g}, \partial_{\xi})}{g(\partial_{\xi} , \partial_{\xi} )^{\frac{1}{2}}} \right) \in [0, \frac{\pi}{2}].
\end{equation}
It is easy to see that the angle quantity is a conformal invariance. Clearly, the $ g $-angle condition defined in (\ref{intro:eqn0}) is conformally invariant, and therefore is natural in the conformal geometry setting. Our geometric result gives the existence of PSC metrics on $ X \cong X \times \lbrace 0 \rbrace_{\xi} $ along with the $ g $-angle condition.
\begin{theorem}\label{intro:thm1}
    Assume that the compact cylinder $ (M, \partial M), \dim M \geqslant 3 $ admits a metric $ g $ with positive scalar curvature $ R_{g} > 0 $ and nonnegative mean curvature $ h_{g} \geqslant 0 $. If $ g $ satisfies the $ g $-angle condition
    \begin{equation}\label{intro:eqn1}
    \angle_{g}(\nu_{g}, \partial_{\xi}) < \frac{\pi}{4} \Leftrightarrow \cos\left(\angle_{g}\left(\nu_{g},\partial_{\xi} \right) \right) > \frac{\sqrt{2}}{2} \Leftrightarrow \frac{g(\partial_{\xi}, \partial_{\xi})}{g(\nu_{g}, \partial_{\xi})^{2}} < 2
    \end{equation}
everywhere on $ X \times \lbrace 0 \rbrace_{\xi} $, then there exists a metric $ \tilde{g} \in [g] $ on $ (M, \partial M) $ such that $ \imath^{*} \tilde{g} $ has positive scalar curvature on $ X \times \lbrace 0 \rbrace_{\xi} \cong X $.
\end{theorem}
\noindent When we consider the geometric assumption (\ref{intro:eqn1}), the PSC metric can be found within the conformal class of the original metric. In addition, it works for every dimension of $ X $ at least $ 2 $.

The trivial example is $ (M, \partial M, k ) $ where $ k = k_{0} \oplus d\xi^{2} $ is a product metric for a metric $ k_{0} $ on $ X $. Clearly $ h_{k} = 0 $. If $ R_{k}  > 0 $, it follows that $ R_{\imath^{*}k} > 0 $. In this situation, $ \nu_{k} = \partial_{\xi} $. Therefore the $ k $-angle $ \angle_{k}(\nu_{k}, \partial_{\xi}) = 0 < \frac{\pi}{4} $. Denote by hypersurfaces $ X_{\xi} : = X \times \lbrace \xi \rbrace, \xi \in [0, 1] $. We can define the second fundamental form, mean curvature, and unit normal vector field $ \nu_{g, \xi} $ along $ X_{\xi} $. Analogously, we can define the $ g $-angle on $ X_{\xi} $ by $ \angle_{g}(\nu_{g, \xi}, \partial_{\xi}) $. If we choose the sign of $ \nu_{g, \xi} $ such that $ \angle_{g}(\nu_{g, \xi}, \partial_{\xi}) \in [0, \frac{\pi}{2}] $, the same argument in Theorem \ref{POS:thm1} implies that the same conclusion in the Theorem holds on $ X_{\xi} \cong X $ provided that
\begin{equation*}
\angle_{g}(\nu_{g, \xi}, \partial_{\xi}) < \frac{\pi}{4} \; {\rm on} \; X_{\xi},
\end{equation*}
as we mentioned in Remark \ref{POS:re1}.

When $ \chi(X) = 0 $, we show that
\begin{equation}\label{intro:eqn00}
   \chi(X) = 0 
   \Rightarrow \exists F : X \times I \rightarrow X \times I, \tilde{\xi} : = \pi_{I} \circ F, \; \rm{s.t.} \; \angle_{g}(\nu_{g, \tilde{\xi}}, \partial_{\tilde{\xi}}) \in [0, \frac{\pi}{4}) \; {\rm on} \; \{ \tilde{\xi} = \frac{1}{2} \},
\end{equation}
such that $ \tilde{\xi} $ is a new global coordinate, and $ \partial_{\tilde{\xi}} $ is a global vector field tangential to the $ I_{\tilde{\xi}} $-direction. Here $ \chi(X) = 0 $ implies the existence of a nowhere vanishing vector field $ V \in \Gamma(TX) $. $ F $ is then constructed by the flow related to $ V $. Since (\ref{intro:eqn00}) holds, we get a PSC metric on $ \{ \tilde{\xi} = \frac{1}{2} \} $ by Theorem \ref{intro:thm1}, we then apply the topological result with surgery for $ \dim X \geq 5 $ \cite[Proposition 6.4]{Rade} to verify Theorem \ref{intro:thm0}. It rules out all potential counterexamples when $ \dim X = 4 $.

The results in \cite{CZ}, \cite{GROMOV2}, \cite{GL}, \cite{Rade},  \cite{Zeidler}, etc.\ use methods from spin geometry, Dirac operators, minimal surfaces, and $ \mu $-bubbles. Our method is different: we introduce an auxiliary one-dimensional space to construct a conformal factor for the original metric $ g $ on $ X \times I $, which comes from the solution of a partial differential equation that is very closely related to the Gauss-Codazzi equation. Our method is also related to the classical methods. Very roughly speaking, both conformal invariance and spin structure observe the tracelessness of stress-energy tensors; both the $ g $-angle condition and $  \mu $-bubble method provide a sense of stability, analogous to the stability of minimal hypersurfaces. 

We now outline why we need an auxiliary dimension in the context of conformal geometry and partial differential equations. The standard way to consider the sign of $ R_{\imath^{*}g} $ on $ X \times \lbrace 0 \rbrace $ is the Gauss-Codazzi equation:
\begin{equation}\label{intro:eqn2}
R_{g} = R_{\imath^{*}g} + 2 Ric_{g}(\nu_{g}, \nu_{g}) - h_{g}^{2} + \lvert A_{g} \rvert^{2}.
\end{equation}
Here $ A_{g} $ is the second fundamental form with respect to $ g $. If we consider a conformal metric $ \tilde{g} = e^{2\phi}g $ on $ (M, \partial M ) $, applying Gauss-Codazzi equation (\ref{intro:eqn2}) and laws of conformal transformations, we have
\begin{equation}\label{intro:eqn3}
\begin{split}
    R_{\imath^{*} \tilde{g}} & = R_{\tilde{g}} - 2Ric_{\tilde{g}}(\nu_{\tilde{g}}, \nu_{\tilde{g}}) + h_{\tilde{g}}^{2} - \lvert A_{\tilde{g}} \rvert^{2} \\
    & = e^{-2\phi}\left(R_{g} - 2Ric_{g}(\nu_{g}, \nu_{g}) + 2(n - 2) \nabla_{\nu_{g}} \nabla_{\nu_{g}} \phi - 2(n - 2) \Delta_{g} \phi + h_{g}^{2} - \lvert A_{g} \rvert^{2} \right) \\
    & \qquad - e^{-2\phi}\left( 2(n - 2) \nabla_{\nu_{g}} \phi \nabla_{\nu_{g}} \phi + (n - 3)(n - 2) \lvert \nabla_{g} \phi \rvert^{2} \right).
\end{split}
\end{equation}
In order to get $ R_{\imath^{*}\tilde{g}} > 0 $, we need to get a solution $ \phi $ of some partial differential equation with second order differential operator $ 2(n - 2) \nabla_{\nu_{g}} \nabla_{\nu_{g}}  - 2(n - 2) \Delta_{g} $, where the inhomogeneous term of the PDE is positive and large enough on $ X \times \lbrace 0 \rbrace $; In addition, we need other $ \phi $-related terms in the conformal transformation of the Gauss-Codazzi equation (\ref{intro:eqn3}) to be small, i.e. we require the solution--which will be used as the conformal factor--to have $ \calC^{1, \alpha} $-smallness. The largeness of the inhomogeneous term on $ X \times \lbrace 0 \rbrace  $ thus dominates all other terms in (\ref{intro:eqn3}) and therefore we obtain a PSC metric $ \tilde{g} $, i.e. $ R_{\imath^{*} \tilde{g}} > 0 $. From this, a second conformal transformation will give a Riemannian metric $ g' $ with $ R_{g'} > 0 $ and $ R_{\imath^{*}g'} > 0 $ by an earlier work of the author \cite{XU7}. 

Unfortunately, the operator $ 2(n - 2) \nabla_{\nu_{g}} \nabla_{\nu_{g}}  - 2(n - 2) \Delta_{g} $ is never elliptic, which means that we may not be able to get a desired solution of such a PDE with desired inhomogeneous term. To resolve this issue, we introduce the auxiliary one-dimensional space, summarized by the following diagram. This codimension two approach is inspired by our recent work of $ \mathbb{S}^{1} $-stability conjecture \cite{RX2}, where the auxiliary space is $ [-1, 1]$. The essential difference is that the auxiliary one-dimensional space we introduce here is $ \mathbb{S}^{1} $ with $ t $-variable and standard Riemannian metric $ dt^{2} $, since we require $ M \times \mathbb{S}^{1} $ to be a manifold with boundary, but not a manifold with corners. Again, we will interchangeably use $ \mathbb{S}^{1} = \mathbb{S}_{t}^{1} $ for clarity.
\begin{equation}\label{diagram}
\begin{tikzcd} (M, g) \arrow[r,"\tau_{1}"] & (M \times \mathbb{S}^{1}_{t},g \oplus dt^{2}) \\
(X \times \lbrace 0 \rbrace_{\xi}, \imath^*g) \arrow[r,"\tau_{2}"]\arrow[u,"\imath"] &  (W = X \times \mathbb{S}^{1}_{t}, \sigma^*(g \oplus dt^{2})) 
\arrow[u,"\sigma"]
\end{tikzcd}
\end{equation}
We always identify $ X \cong X \times \lbrace 0 \rbrace_{\xi} \subset M $ and $ W \cong W \times \lbrace 0 \rbrace_{\xi} \subset M \times \mathbb{S}^{1}_{t} $. Here $ \sigma : W \rightarrow M \times \mathbb{S}^{1}_{t} $ is given by $ \sigma(w) = (w, 0) $ with fixed point $ 0 \in I_{\xi} $. The two inclusions $ \tau_{1}, \tau_{2} $ are sending $ M \ni x \mapsto (x, P) = \tau_{1}(x), X \times \lbrace 0 \rbrace_{\xi} \ni y \mapsto (y, P) = \tau_{2}(y) $ for some fixed point $ P \in \mathbb{S}^{1} $. With natural projections $ \Pi_{1} : M \times \mathbb{S}^{1}_{t} \rightarrow M $ and $ \Pi_{2} : W \times \lbrace 0 \rbrace_{\xi} \rightarrow X \times \lbrace 0 \rbrace_{\xi} \times \lbrace P \rbrace $, we get $ \Pi_{i} \circ \tau_{i} = Id, i = 1, 2 $. Under local parametrization, we may identify $ P $ with point $ 0 $ in any local chart containing $ P $. We use the labels $ \lbrace 0 \rbrace_{\xi} \in I_{\xi} $ and $ \lbrace 0 \rbrace_{t} \cong \lbrace P \rbrace \in \mathbb{S}_{t}^{1} $ to distinguish these two points.

With the help of the auxiliary space, we construct a partial differential equation with a modified, elliptic differential operator on the closed manifold $ W $. The crucial point is that the PDE is elliptic when the g-angle condition holds. Such a PDE with desired inhomogenous term admits a $ \calC^{1, \alpha} $-small solution. We then lift a slight modification of this solution to the ambient space $ M \times \mathbb{S}^{1}_{t} $, serving as a conformal factor of the metric $ \bar{g} = g \oplus dt^{2} $ on $ M \times \mathbb{S}^{1}_{t} $. When we restrict this conformal transformation to $ M $, the conformal factor plays a role in transferring some geometric data from $ W $ to $ M $. With some work, the conformal factor and the original geometric information together yield the desired geometric information on $ X \times \lbrace 0 \rbrace_{\xi} \cong X $.

As an outline of this article, we construct the conformal factor $ u' : M \times \mathbb{S}_{t}^{1} \rightarrow \R $ in \S2 via a series of technical results. In \S2, we also verify some properties the conformal factor $ u' $ must satisfy, especially its $ \calC^{0} $-estimates and partial $ \calC^{2} $-estimates. The $ g $-angle condition is used to obtain the ellipticity of the differential operator, which comes from the conformal transformation of the Gauss-Codazzi equation.

In \S3 we prove the main result with $ g $-angle condition in Theorem~\ref{POS:thm1}, essentially by using Gauss-Codazzi on $ \tilde{g} = (u')^{\frac{4}{n - 2}} \bar{g} $ to first move from $M \times \mathbb{S}_{t}^{1} $ back to $ M $ and then back to $X \cong X \times \lbrace 0 \rbrace_{\xi} $. Corollary \ref{POS:cor2} generalizes the results of Gromov and Lawson \cite{GL} and R\"ade \cite{Rade}. With a direct application of a Yamabe-type problem with Dirichlet boundary condition \cite{XU7}, we show in Corollary \ref{POS:cor1} that there exists a metric $ g' $ in the conformal class $ [g] $ such that $ R_{g'} > 0 $ and $ R_{\imath^{*}g'} > 0 $ under the same hypotheses of Theorem \ref{POS:thm1}.

In \S4, we construct a new smooth hypersurface $ S $ of $ M $ on which the angle condition holds in Proposition \ref{Euler:prop1}. We then show in Lemma \ref{Euler:lemma2} that there exists one connected component of $ S $ that separates $ X \times \{ 0 \} $ and $ X \times \{ 1 \} $. We then give the main result with $ \chi(X) = 0 $ hypothesis in Theorem~\ref{Euler:thm2}. The classical result of Gromov and Lawson for scalar and mean curvature comparison is further generalized in Corollary \ref{Euler:cor1}.
\medskip

The author would like to thank Boris Botvinnik, Robert McOwen, Steven Rosenberg and Junrong Yan for many discussions on this topic.

\section{Technical Results}
In this section, we give basic setup and all technical results that are essential for our main result. These technical results are analogous to the technical results of \cite[Section 2]{RX2}. Recall that $ X $ is a closed, oriented manifold, and $ M = X \times I_{\xi} $ is a $ n $-dimensional compact cylinder with Riemannian metric $ g $ such that $ R_{g} > 0 $ and $ h_{g} \geqslant 0 $. Without loss of generality, we assume, throughout this article, that $ R_{g} > 0 $ and $ h_{g} \equiv 0 $. This can be done by taking a conformal transformation of the original metric, see \cite{ESC}. Throughout this article, we denote by $ \bar{g} = g \oplus dt^{2} $ the Riemannian metric on $ M \times \mathbb{S}^{1}_{t} $. It follows that $ R_{\bar{g}} > 0 $ on $ M \times \mathbb{S}_{t}^{1} $ and $ h_{\bar{g}} \equiv 0 $ on $ \partial( M \times \mathbb{S}_{t}^{1}) = X \times \lbrace 0 \rbrace_{\xi} \times \mathbb{S}_{t}^{1} \cup X \times \lbrace 1 \rbrace_{\xi} \times \mathbb{S}_{t}^{1} $. We also assume the familiary of the standard knowledge of Sobolev space $ W^{k, p} $, and $ \calL^{p} $-elliptic regularity, see e.g. \cite{Niren4}. For any metric $ g_{0} $, we set $ -\Delta_{g_{0}} $ to be the positive definite Laplace-Beltrami operator.

Choosing local coordinates $ \lbrace x^{i}, x^{n} = \xi \rbrace $ of $ M $ and associated local frame $ \lbrace \partial_{i}, \partial_{n} = \partial_{\xi} \rbrace $ such that $ \lbrace \partial_{i} \rbrace_{i = 1}^{n - 1} $ are tangential to $ \partial M $, we can write
\begin{equation*}
    \nu_{g} = a\partial_{\xi} + \sum_{i = 1}^{n - 1} b^{i} \partial_{i}.
\end{equation*}
This defines a $ g $-related vector field $ V = \nu_{g} - a\partial_{\xi} $, which is locally expressed by $ V = \sum_{i = 1}^{n - 1} b^{i} \partial_{i} $. Note that
\begin{equation*}
1 = g(\nu_{g}, \nu_{g}) = g(a\partial_{\xi}, \nu_{g}) + \sum_{i} g(b^{i} \partial_{i}, \nu_{g})  \Rightarrow a = g(\partial_{\xi}, \nu_{g})^{-1}.
\end{equation*}
$ a $ is nowhere vanishing since both $ \nu_{g} $ and $ \partial_{\xi} $ are nowhere tangent to $ \partial M $.

Clearly $ g(V, \nu_{g}) = g(\nu_{g} - a\partial_{\xi}, \nu_{g}) = 0 $, thus the vector field $ V $ is tangential to $ \partial M $, and can be extended to $ (W \cong X \times \lbrace 0 \rbrace _{\xi} \times \mathbb{S}^{1}_{t}, \sigma^{*}\bar{g}) $ by $T_{x}(X) \ni V_{x} \mapsto (V_{x}, 0) \in T_{(x, \xi)}(X \times I) $.
We still denote the latter extended vector field by $ V $ as a global vector field on $ W $. We apply the same argument in \cite[Lemma 2.1]{RX2} to show that:
\begin{lemma}\label{Set:lemma1} If 
\begin{equation}\label{Set:eqn0}
\frac{g(\partial_{\xi},\partial_{\xi})}{g(\nu_{g},\partial_{\xi})^{2}}<
2,
\end{equation}
 then 
the operator
\begin{equation*}
 L':=\nabla_{V}\nabla_{V} - \Delta_{\sigma^{*}\bar{g}}
\end{equation*}
is elliptic on $ W $.
\end{lemma}
\begin{proof} The same argument as in \cite[Lemma 2.1]{RX2} follows.
\end{proof}

The next three lemmas and propositions are crucial to construct the conformal factor we will use in the main theorem. For the standard $ dt^{2} $-metric on $ \mathbb{S}^{1}_{t} $, we fix some chart $ U \ni P $ with chart map $ \Phi : \mathbb{S}^{1} \supset U \rightarrow (-1, 1) $ such that $ \Phi(P) = 0 $ and still with locally $ t $-variable. We start with the construction of the inhomogeneous term of our partial differential equation.
\begin{lemma}\label{Set:lemma2} 
For any  $ p \in \mathbb{N} $,  $ C \gg 1 $, and any positive $ \delta \ll 1 $.  
Then there exists a positive smooth function $F:W\to\R$ and small enough constant $ \epsilon \ll 1 $ such that $ F |_{X \times (-\frac{\epsilon}{2}, \frac{\epsilon}{2})_{t}} = C+1  $, $ F \equiv 0 $ outside $ X \times [-\epsilon, \epsilon]_{t} $ and $ \lVert  F \rVert_{\calL^{p}(W, \sigma^{*}\bar{g})} < \delta. $
\end{lemma}
\begin{proof}
Fix any $ p, C $ and $ \delta $. We take a constant function $ f = C + 1 : W \rightarrow \R $. There exists a nonnegative smooth function $  \phi :\mathbb{S}^{1}_{t} \rightarrow R $ such that 
\begin{equation*}
  \phi(t) \equiv 1, t \in \left(-\frac{\epsilon}{2}, \frac{\epsilon}{2}\right), \phi(t) = 0, t \notin [-\epsilon, \epsilon].
\end{equation*}
Set $ F = f \cdot \phi : W \rightarrow \R $. It follows that
\begin{equation*}
    \lVert F \rVert_{\calL^{p}(W, \sigma^{*}\bar{g})} < \delta
\end{equation*}
provided that $ \epsilon \ll 1 $ is small enough.
\end{proof}
\begin{remark}\label{Set:re0}
The function $ F $ in any local frame $ \lbrace \partial_{1}, \dotso, \partial_{n - 1}, \partial_{n + 1} = \partial_{t} \rbrace $ of $ W $ satisfies
\begin{equation*}
\partial_{i} F \equiv 0, \forall i \in \lbrace 1, \dotso, n - 1 \rbrace
\end{equation*}
based on the construction of $ F $ in Lemma \ref{Set:lemma2}.
\end{remark}
We now construct the first candidate of our conformal factor with small enough $ \calC^{1, \alpha} $-norm. We define the $\calC^{1, \alpha} $-norm on $ W $ by fixing a chart cover $ W = \bigcup_{i} (U_{i}, \varphi_{i}) $ such that
\begin{equation*}
    \lVert u \rVert_{\calC^{1, \alpha}(W)} = \lVert u \rVert_{\calC^{0}(W)} + \sup_{i, k} \lVert \partial_{x_{i}^{k}}u \rVert_{\calC^{0}(U_{i})} + \sup_{i, k} \sup_{x \neq x', x, x' \in U_{i}} \frac{\left\lvert \partial_{x_{i}^{k}}u(x) - \partial_{x_{i}^{k}}u(x') \right\rvert}{\lvert x - x' \rvert^{\alpha}}
\end{equation*}
where $ \partial_{x_{i}^{k}} u $ are local representations in $ U_{i} $ with local coordinates $ \lbrace x_{i}^{1}, \dotso, x_{i}^{n - 1}, x_{i}^{n + 1} = t \rbrace $ and associated local frames $ \lbrace \partial_{x_{i}^{k}} \rbrace $. 
\begin{proposition}\label{Set:prop1}
Let $ (W, \sigma^{*} \bar{g}) $ be as above. Assume that (\ref{Set:eqn0}) holds. For any positive constant $ \eta \ll 1 $, any positive constant $ C $, and any $ p > n = \dim W $, there exists an associated $ F $ and $ \delta $ in the sense of Lemma \ref{Set:lemma2}, such that the following partial differential equation     
\begin{equation}\label{Set:eqn3}
L u : = 4\nabla_{V} \nabla_{V} u - 4\Delta_{\sigma^{*} \bar{g}}  u + R_{\bar{g}} |_{W}u =
F  \; {\rm in} \; W
\end{equation}
admits a unique smooth solution $ u $ with
\begin{equation}\label{Set:eqn4}
\lVert u \rVert_{\calC^{1, \alpha}(W)} < \eta
\end{equation}
for some $ \alpha \in (0, 1) $ such that $ \alpha \leq 1 - \frac{n}{p} $.
\end{proposition}
\begin{proof}
Note that $ W $ is a closed, oriented Riemannian manifold. Fix $ \eta \ll 1, C $, and $ p > n $. We then fix some $ \alpha > 0 $ such that $ 1 + \alpha \leqslant 2 - \frac{n}{p} $. Denote $ \bar{C} = C(W, g, n, p, L) $ by the constant of $ \calL^{p} $ elliptic regularity estimates with respect to $ L $, and $ C' = C'(W, g, n, p, \alpha) $ by the constant of the Sobolev embedding $ W^{2, p} \hookrightarrow \calC^{1, \alpha} $. Fix $ \delta $ such that $ \delta \bar{C} C' < \eta $. Finally we choose an associated $ F $ in the sense of Lemma \ref{Set:lemma2}.

By assumption, $ R_{g} > 0 $ on $ M $, therefore $ R_{\bar{g}} > 0 $ on $ M \times \mathbb{S}^{1}_{t} $, hence $ R_{\bar{g}} |_{W} > 0 $ uniformly. By Lemma \ref{Set:lemma1}, the operator in (\ref{Set:eqn3}) is elliptic. By maximum principle on closed manifolds, the elliptic operator is an injective operator. It follows that there exists a unique solution $ u \in H^{1}(W, \sigma^{*}\bar{g}) $ of (\ref{Set:eqn3}) by Fredholm alternative. By standard $ H^{s} $-type elliptic theory and the smoothness of $ F $, $ u \in \calC^{\infty}(W) $. 

By standard $ \calL^{p} $-regularity theory \cite{Niren4}, 
\begin{equation*}
    \lVert u \rVert_{W^{2, p}(W, \sigma^{*} \bar{g})} \leqslant \bar{C} \left( \lVert F \rVert_{\calL^{p}(W, \sigma^{*} \bar{g})} + \lVert u \rVert_{\calL^{p}(W, \sigma^{*}\bar{g})} \right).
\end{equation*}
Due to the injectivity of the operator, a very similar argument of \cite[Proposition 2.1]{RX2} shows that the $ \calL^{p}$ estimates can be improved by
\begin{equation*}
     \lVert u \rVert_{W^{2, p}(W, \sigma^{*} \bar{g})} \leqslant \bar{C} \lVert F \rVert_{\calL^{p}(W, \sigma^{*} \bar{g})}.
\end{equation*}
By Sobolev embedding, it follows that
\begin{equation*}
    \lVert u \rVert_{\calC^{1, \alpha}(W)} \leqslant C' \lVert u \rVert_{W^{2, p}(W, \sigma^{*}\bar{g})} \leqslant \bar{C}C'\lVert F \rVert_{\calL^{p}(W, \sigma^{*} \bar{g})} < \bar{C}C' \delta < \eta. 
\end{equation*}
\end{proof}
\begin{remark}\label{Set:re1}
By maximum principle and the nonnegativity of $ F $, the solution of (\ref{Set:eqn3}) satisfies $ u \geqslant 0, u \not\equiv 0 $ in $ W $.
\end{remark}
We now give our candidate of the conformal factor on $ M \times \mathbb{S}^{1}_{t} $. To begin with, 
we define a new function
\begin{equation*}
    u_{0} : = u + 1, u_{0} : W \rightarrow \R
\end{equation*}
where $ u $ is the solution of (\ref{Set:eqn3}). With the natural projection $ \pi : M \times \mathbb{S}^{1}_{t} \rightarrow W $, we can pullback $ u_{0} $ to $ M \times \mathbb{S}_{t}^{1} $ by $ u' : = (\pi)^{*} u_{0} $. We also pullback $ u $ to $ M \times \mathbb{S}_{t}^{1} $ by $ \bar{u} : = (\pi)^{*} u $ for later use. We have
\begin{equation}\label{Set:eqn6}
u' |_{W} =u_{0} = u + 1, \partial_{\xi}^{l} u' = 0 \; {\rm on} \; M \times \mathbb{S}^{1}_{t}, \forall l \in \mathbb{Z}_{> 0}.
\end{equation}
Denote $ \tilde{u}' : = \tau_{1}^{*} u' : M \rightarrow \R $. Following our diagram (\ref{intro:eqn2}), we have
\begin{equation*}
    \imath^{*} \tilde{u}' = \imath^{*} \tau_{1}^{*} u' = \tau_{2}^{*} \sigma^{*} u' = \tau_{2}^{*} u_{0} = u_{0} |_{X \times \lbrace 0 \rbrace_{\xi} \times \lbrace P \rbrace_{t}}.
\end{equation*}
Clearly all those functions defined above are positive functions on corresponding spaces. The function $ (u')^{\frac{4}{n - 2}} $ will be our choice of conformal factor on $ M \times \mathbb{S}_{t}^{1} $. The relations above hold analogously for $ \bar{u} $ also.
\medskip

Usually, the best estimate of the solution of (\ref{Set:eqn3}) is given by (\ref{Set:eqn4}). Due to the special structure of the product metric $ \bar{g} $ and the construction of $ F $, we are able to give a partial $ \calC^{2} $-estimate in terms of $ \left\lVert \frac{\partial^{2} u}{\partial t^{2}} \right\rVert_{\calC^{0}(X \times \lbrace 0 \rbrace_{\xi} \times \lbrace P \rbrace_{t})} $, as the next lemma shows.
\begin{lemma}\label{Set:lemma3}
Choosing $ \epsilon $ defined in Lemma \ref{Set:lemma2} to be small enough that will be determined below. Let $ (W, \sigma^{*} \bar{g}) $, $ p, C, \delta $ and $ F $ be the same as in Proposition \ref{Set:prop1}. Let $ u $ be the associated solution of (\ref{Set:eqn3}). There exists a constant $ \eta' \ll 1 $ such that
\begin{equation}\label{Set:eqn5}
    \left\lVert \frac{\partial^{2} u}{\partial t^{2}} \right\rVert_{\calC^{0}\left(X \times \lbrace 0 \rbrace_{\xi} \times \left(-\frac{\epsilon}{4}, \frac{\epsilon}{4} \right)_{t} \right)} < \eta'.
\end{equation}
\end{lemma}
\begin{proof}
Set $ U_{1} = X \times \lbrace 0 \rbrace_{\xi} \times \left(-\frac{\epsilon}{4}, \frac{\epsilon}{4} \right) $, $ U_{2} =  X \times \lbrace 0 \rbrace_{\xi} \times \left(-\frac{\epsilon}{2}, \frac{\epsilon}{2} \right) $, $ O_{1} =  X \times \lbrace 0 \rbrace_{\xi} \times \left(-\frac{1}{4}, \frac{1}{4} \right) $, $ O_{2} =  X \times \lbrace 0 \rbrace_{\xi} \times \left(-\frac{1}{2}, \frac{1}{2} \right) $. Since $ \bar{g} = g \oplus dt^{2} $, the vector field $ V $ and $ \Delta_{\sigma^{*} \bar{g}} $ are constant with respect to $ t $; furthermore the scalar curvature $ R_{\bar{g}} $ is constant along $ t $-direction. In addition, $ F \equiv C + 1 $ on $ U_{2} $ by Lemma \ref{Set:lemma2}, we apply $ \frac{\partial^{2}}{\partial t^{2}} $ on both sides of (\ref{Set:eqn3}) in $ U $, which follows that
\begin{equation*}
  L_{\sigma^{*} \bar{g}} \left( \frac{\partial^{2}u}{\partial t^{2}} \right) : =  4\nabla_{V} \nabla_{V} \left( \frac{\partial^{2}u}{\partial t^{2}} \right) - 4\Delta_{\sigma^{*} \bar{g}}  \left( \frac{\partial^{2}u}{\partial t^{2}} \right) + R_{\bar{g}} |_{W}\left( \frac{\partial^{2}u}{\partial t^{2}} \right) = 0  \; {\rm in} \; U_{1}.
\end{equation*}
Now consider $ t = \epsilon t' $, the new $ t' $-variable is associated with the metric $ (dt')^{2} = \epsilon^{-2} dt^{2} $ on $ \mathbb{S}^{1} $. We have
\begin{equation*}
   (\epsilon^{-1})^{2} \left( 4\nabla_{V} \nabla_{V} \left( \frac{\partial^{2}u}{\partial (t')^{2}} \right) - 4\Delta_{\sigma^{*} \bar{g}}  \left( \frac{\partial^{2}u}{\partial (t')^{2}} \right) + R_{\bar{g}} |_{W}\left( \frac{\partial^{2}u}{\partial (t')^{2}} \right) \right)= 0  \; {\rm in} \; O_{1}.
\end{equation*}
Equivalently, the following PDE holds in $ O_{1} $ with the new metric $ \bar{g} \mapsto \epsilon^{-2} \bar{g} $,
\begin{equation}\label{Set:eqn6a}
L_{\sigma^{*}(\epsilon^{-2}\bar{g})}\left( \frac{\partial^{2}u}{\partial (t')^{2}} \right) : = 4\nabla_{V_{\epsilon^{-2}\bar{g}}} \nabla_{V_{\epsilon^{-2}\bar{g}}} \left( \frac{\partial^{2}u}{\partial (t')^{2}} \right) - 4\Delta_{\sigma^{*} (\epsilon^{-2}\bar{g})}  \left( \frac{\partial^{2}u}{\partial (t')^{2}} \right) + R_{\epsilon^{-2}\bar{g}} |_{W}\left( \frac{\partial^{2}u}{\partial (t')^{2}} \right) = 0  \; {\rm in} \; O_{1}.
\end{equation}
Here $ V_{\epsilon^{-2}\bar{g}} $ is the vector field with respect to the new metric $ \epsilon^{-2} \bar{g} $ satisfying $ V_{\epsilon^{-2}\bar{g}} = \epsilon V $.

We now estimate $ \lVert \frac{\partial^{2}u}{\partial (t')^{2}} \rVert_{\calC^{0}(O_{1})} $. We set $ s > 0 $ such that $ s - \frac{n}{2} = 1 + \alpha \leqslant 2 - \frac{n}{p} $, where $ p $ and $ \alpha $ are given in Proposition \ref{Set:prop1}. With usual metric $ \sigma^{*} \bar{g} $, we have $ H^{s}$-type local elliptic regularity estimates for second order elliptic operator $ L_{\sigma^{*}\bar{g}} $,
\begin{equation*}
\lVert v \rVert_{H^{s}(U, \sigma^{*} \bar{g})} \leqslant D_{s} \left( \lVert L_{\sigma^{*} \bar{g}} u \rVert_{H^{s - 2}(V, \sigma^{*} \bar{g})} + \lVert v \rVert_{\calL^{2}(V, \sigma^{*} \bar{g})} \right), \forall v \in \calC^{\infty}(V).
\end{equation*}
where the constant $ D_{s} $ only depends on the operator $ L_{\sigma^{*}\bar{g}} $, the metric, the degree $ s $ and the domain $ U, V $. For the scaled metric $ \sigma^{*} (\epsilon^{-2} \bar{g}) $, we apply the local estimates and the second order elliptic operator $ L_{\sigma^{*}(\epsilon^{-2}\bar{g})} $ in (\ref{Set:eqn6a}), 
\begin{equation}\label{Set:eqn70a}
\begin{split}
\left\lVert \frac{\partial^{2}u}{\partial (t')^{2}} \right\rVert_{H^{s}(O_{1}, \sigma^{*}(\epsilon^{-2} \bar{g}))} & \leqslant \epsilon^{-\frac{n}{2}} \left\lVert \frac{\partial^{2}u}{\partial (t')^{2}} \right\rVert_{H^{s}(O_{1}, \sigma^{*} \bar{g})} \\
& \leqslant \epsilon^{-\frac{n}{2}} D_{s} \left\lVert \frac{\partial^{2}u}{\partial (t')^{2}} \right\rVert_{\calL^{2}(O_{2}, \sigma^{*} \bar{g})} \leqslant  D_{s} \left\lVert \frac{\partial^{2}u}{\partial (t')^{2}} \right\rVert_{\calL^{2}(O_{2}, \sigma^{*}(\epsilon^{-2} \bar{g}))}.
\end{split}
\end{equation}
Applying the local $ H^{2} $-type local elliptic estimate for (\ref{Set:eqn3}) and H\"older's inequality,
\begin{equation}\label{Set:eqn7a}
    \begin{split}
        \left\lVert \frac{\partial^{2}u}{\partial (t')^{2}} \right\rVert_{\calL^{2}(O_{2}, \sigma^{*} (\epsilon^{-2} \bar{g}))} & = \epsilon^{2 - \frac{n}{2}} \left\lVert \frac{\partial^{2}u}{\partial t^{2}} \right\rVert_{\calL^{2}(O_{2}, \sigma^{*} \bar{g})} \leqslant  \epsilon^{2 - \frac{n}{2}} \lVert u \rVert_{H^{2}(O_{2}, \sigma^{*}\bar{g})} \\
        & \leqslant \epsilon^{2 - \frac{n}{2}} D_{2} \left( \lVert F \rVert_{\calL^{2}(W, \sigma^{*}\bar{g})} + \lVert u \rVert_{\calL^{2}(W, \sigma^{*}\bar{g})} \right) \\
        & \leqslant \epsilon^{2 - \frac{n}{2}}D_{2}\lVert F \rVert_{\calL^{2}(W, \sigma^{*}\bar{g})} + \epsilon^{2 - \frac{n}{2}}D_{1} D_{2}\lVert u \rVert_{\calL^{\frac{2(n + 1)}{n - 1}}(W, \sigma^{*}\bar{g})} \\
        & \leqslant \epsilon^{2} D_{2}  \lVert F \rVert_{\calL^{2}(W, \sigma^{*}(\epsilon^{-2}\bar{g}))} + \epsilon^{2 - \frac{n}{n + 1}} D_{1}D_{2} \lVert u \rVert_{\calL^{\frac{2(n + 1)}{n - 1}}(W, \sigma^{*} (\epsilon^{-2} \bar{g}))}.
    \end{split}
\end{equation}
Similarly, the Sobolev embedding inequality for the scaled metric $ \sigma^{*} (\epsilon^{-2} \bar{g}) $ says
\begin{equation}\label{Set:eqn7}
\left\lVert \frac{\partial^{2}u}{\partial (t')^{2}} \right\rVert_{\calC^{0}(O_{1})} \leqslant D_{0} \epsilon^{\frac{n}{2}} \left\lVert \frac{\partial^{2}u}{\partial (t')^{2}} \right\rVert_{H^{s}(O_{1}, \sigma^{*}(\epsilon^{-2} \bar{g}))}.
\end{equation}
Here $ D_{0}, D_{1}, D_{2} $ are independent of $ \epsilon $ and $ u $.

On $ M \times \mathbb{S}^{1}_{t} $ with $ \dim (M \times \mathbb{S}^{1}_{t} ) = n + 1 $, the metric $ \bar{g} $ has positive relative Yamabe constant $ \lambda : = \lambda(M \times \mathbb{S}^{1}_{t}, \partial (M \times \mathbb{S}^{1}_{t}), [\bar{g}]) > 0 $. It follows from the definition of the relative Yamabe constant that for all nontrivial $ v \in \calC^{\infty}(M \times \mathbb{S}_{t}^{1}) $, the Yamabe quotient gives
\begin{align*}
\lVert v \rVert_{\calL^{\frac{2(n + 1)}{n - 1}}(M \times \mathbb{S}^{1}_{t}, \bar{g})}^{2} & \leqslant \lambda(M \times \mathbb{S}^{1}_{t}, \partial (M \times \mathbb{S}^{1}_{t}), [\bar{g}])^{-1} \left( \frac{4n}{n - 1} \lVert \nabla_{\bar{g}} v \rVert_{\calL^{2}(M \times \mathbb{S}^{1}_{t}, \bar{g})}^{2}  \right) \\
& \qquad + \lambda(M \times \mathbb{S}^{1}_{t}, \partial (M \times \mathbb{S}^{1}_{t}), [\bar{g}])^{-1} \left( \int_{M \times \mathbb{S}^{1}_{t}} R_{\bar{g}} v^{2} d\text{Vol}_{\bar{g}} + 2n \int_{\partial (M \times \mathbb{S}^{1}_{t})} h_{\bar{g}} v^{2} \sigma_{\bar{g}} \right) \\
& = \lambda^{-1} \left( \frac{4n}{n - 1} \lVert \nabla_{\bar{g}} v \rVert_{\calL^{2}(M \times \mathbb{S}^{1}_{t}, \bar{g})}^{2} + \int_{M \times \mathbb{S}^{1}_{t}} R_{\bar{g}} v^{2} d\text{Vol}_{\bar{g}} \right). 
\end{align*}
Here $ \sigma_{\bar{g}} $ is the volume form on the hypersurface $ \partial(M \times \mathbb{S}_{1}^{t}) $. The boundary term is dropped off since we have assumed that $ h_{\bar{g}} \equiv 0 $. The relative Yamabe constant $ \lambda $ is invariant under any conformal transformation of the metric $ \bar{g} $.

Since $ M \times \mathbb{S}_{t}^{1} $ is compact, there exist two positive constants $ D_{3} $ and $ D_{4} $ such that two metrics $ \bar{g} $ and $ \sigma^{*} \bar{g} \oplus d\xi^{2} $ satisfy
\begin{equation*}
   D_{4}^{-2} \sqrt{\det{\bar{g}}} \leqslant \sqrt{\det(\sigma^{*} \bar{g} \oplus d\xi^{2})} \leqslant D_{3}^{\frac{2(n + 1)}{n - 1}} \sqrt{\det{\bar{g}}}.
\end{equation*}
Set $ \xi' = \epsilon^{-1} \xi $, we have
\begin{equation*}
   D_{4}^{-2} \sqrt{\det{(\epsilon^{-2}\bar{g})}} \leqslant \sqrt{\det(\sigma^{*}(\epsilon^{-2} \bar{g}) \oplus d(\xi')^{2})} \leqslant D_{3}^{\frac{2(n + 1)}{n - 1}} \sqrt{\det{(\epsilon^{-2}\bar{g})}}
\end{equation*}
for the same $ D_{3}, D_{4} $. 

Denote the inclusion by $ \sigma_{\xi} : X_{\xi} \times \mathbb{S}^{1} \rightarrow M \times \mathbb{S}^{1} $, and the space $ X_{\xi} \times \mathbb{S}^{1} $ by $ W_{\xi} $. We now apply the three inequalities above for the function $ u $ on $ W $ and $ \bar{u} = \pi^{*}u $ on $ M \times \mathbb{S}_{t}^{1} $ with respect to the metrics $ \sigma^{*}(\epsilon^{-2} \bar{g}) $ and $ \epsilon^{-2} \bar{g} $ respectively. Recall that the function $ \bar{u} $ is constant along each $ \xi $-fiber, it follows that
\begin{align*}
& \lVert u \rVert_{\calL^{\frac{2(n + 1)}{n - 1}}(W, \sigma^{*} (\epsilon^{-2} \bar{g}))} \\
& \qquad = \epsilon^{\frac{n -1}{2(n + 1)}} \left( \int_{0}^{\epsilon^{-1}} \int_{W} \lvert \bar{u} \rvert^{\frac{2(n + 1)}{n - 1}} d\text{Vol}_{\sigma^{*} (\epsilon^{-2} \bar{g})} d\xi'  \right)^{\frac{n - 1}{2(n + 1)}} \\
& \qquad = \epsilon^{\frac{n -1}{2(n + 1)}} \left( \int_{M \times \mathbb{S}_{t'}^{1}} \lvert \bar{u} \rvert^{\frac{2(n + 1)}{n - 1}} d\text{Vol}_{\sigma^{*} \bar{g} \oplus d(\xi')^{2}}  \right)^{\frac{n - 1}{2(n + 1)}} \leqslant \epsilon^{\frac{n -1}{2(n + 1)}} D_{3}  \left( \int_{M \times \mathbb{S}^{1}_{t'}}  \lvert \bar{u} \rvert^{\frac{2(n + 1)}{n - 1}} d\text{Vol}_{\epsilon^{-2}\bar{g}} \right)^{\frac{n - 1}{2(n + 1)}} \\
& \qquad \leqslant \epsilon^{\frac{n -1}{2(n + 1)} } D_{3} \lambda^{-1} \left( \frac{4n}{n - 1} \lVert \nabla_{\epsilon^{-2}\bar{g}} \bar{u} \rVert_{\calL^{2}(M \times \mathbb{S}^{1}_{t'}, \epsilon^{-2}\bar{g})} + \int_{M \times \mathbb{S}^{1}_{t'}} R_{\epsilon^{-2}\bar{g}} \bar{u}^{2} d\text{Vol}_{\epsilon^{-2}\bar{g}} \right)^{\frac{1}{2}} \\
& \qquad \leqslant \epsilon^{\frac{n -1}{2(n + 1)}} D_{3} D_{4} \lambda^{-1} \left( \frac{4n}{n - 1}\int_{0}^{\epsilon^{-1}} \int_{W_{\xi}} \lvert \nabla_{\sigma_{\xi}^{*}(\epsilon^{-2}\bar{g})} \bar{u} \rvert^{2} d\text{Vol}_{\sigma_{\xi}^{*}(\epsilon^{-2}\bar{g})} d\xi' \right)^{\frac{1}{2}} \\
& \qquad \qquad + \epsilon^{\frac{n -1}{2(n + 1)}} D_{3} D_{4} \lambda^{-1} \left( \int_{0}^{\epsilon^{-1}} \int_{W} R_{\epsilon^{-2}\bar{g}} \bar{u}^{2} d\text{Vol}_{\sigma^{*} (\epsilon^{-2}\bar{g})} d\xi' \right)^{\frac{1}{2}} \\
& \qquad \leqslant 2D_{3} D_{4} \lambda^{-1} \epsilon^{-\frac{1}{n + 1}} \left( \frac{4n}{n - 1} \max_{\xi \in [0, 1]} \left( \lVert \nabla_{\sigma_{\xi}^{*} (\epsilon^{-2} \bar{g})}\bar{u} \rVert_{\calL^{2}(W_{\xi}, \sigma_{\xi}^{*}(\epsilon^{-2}\bar{g}))}^{2} \right) + \epsilon^{2} \max_{W} \lvert R_{\bar{g}} \rvert \lVert u \rVert_{\calL^{2}(W, \sigma^{*}(\epsilon^{-2}\bar{g}))}^{2} \right)^{\frac{1}{2}}
\end{align*}
Note that $ \lVert F \rVert_{\calL^{p}(W, \sigma^{*}\bar{g})} $ is small in the sense that $ (C + 1)^{p} \epsilon \ll 1 $, hence $ (C + 1)^{2} \epsilon \ll 1 $ since $ p > n \geqslant 3 $. Note also that there are trivial diffeomorphisms between $ W $ and $ W_{\xi} $ for every $ \xi \in [0, 1] $. Since $ \bar{u} $ is constant for each $ \xi $-fiber, and $ M \times \mathbb{S}^{1} $ is compact, we have $ \lVert \nabla_{\sigma_{\xi}^{*}\bar{g}} \bar{u} \rVert_{\calL^{2}(W_{\xi}, \sigma_{\xi}^{*} \bar{g})} \leqslant D_{5} \eta $ for some $ D_{5} $ independent of $ \bar{u} $ and $ \epsilon $. Without loss of generality, we may assume $ \max_{W} \lvert R_{\bar{g}} \rvert \leqslant D_{5} $ by increasing $ D_{5} $ if necessary. Combining (\ref{Set:eqn4}), (\ref{Set:eqn70a}), (\ref{Set:eqn7a}), (\ref{Set:eqn7}) with the above inequality, we have
\begin{align*}
\left\lVert \frac{\partial^{2}u}{\partial (t')^{2}} \right\rVert_{\calC^{0}(O_{1})} & \leqslant D_{0}\epsilon^{\frac{n}{2}} \lVert \frac{\partial^{2}u}{\partial (t')^{2}} \rVert_{H^{s}(O_{1}, \sigma^{*}(\epsilon^{-2} \bar{g}))} \leqslant D_{0} D_{s} \epsilon^{\frac{n}{2}} \left\lVert \frac{\partial^{2}u}{\partial (t')^{2}} \right\rVert_{\calL^{2}(O_{2}, \sigma^{*}(\epsilon^{-2} \bar{g}))} \\
& \leqslant \epsilon^{2 + \frac{n}{2}} D_{0}D_{2}D_{s}  \lVert F \rVert_{\calL^{2}(W, \sigma^{*}(\epsilon^{-2}g))} + \epsilon^{2 - \frac{n}{n + 1} + \frac{n}{2}} D_{0}D_{1}D_{2}D_{s} \lVert u \rVert_{\calL^{\frac{2(n + 1)}{n - 1}}(W, \sigma^{*} (\epsilon^{-2} \bar{g}))} \\
& \leqslant 2D_{0} D_{1} D_{2} D_{3}D_{4} D_{s} \lambda^{-1}  \epsilon^{1 + \frac{n}{2}} \times \\
& \qquad \times \left( \frac{4n}{n - 1} \max_{\xi \in [0, 1]} \left( \lVert \nabla_{\sigma_{\xi}^{*} (\epsilon^{-2} \bar{g})} \bar{u} \rVert_{\calL^{2}(W_{\xi}, \sigma_{\xi}^{*}(\epsilon^{-2}\bar{g}))}^{2} \right) + \epsilon^{2} \max_{W} \lvert R_{\bar{g}} \rvert \lVert u \rVert_{\calL^{2}(W, \sigma^{*}(\epsilon^{-2}\bar{g}))}^{2} \right)^{\frac{1}{2}} \\
& \qquad \qquad + \epsilon^{2 + \frac{n}{2}} D_{0}D_{2}D_{s}  \lVert F \rVert_{\calL^{2}(W, \sigma^{*}(\epsilon^{-2}\bar{g}))} \\
& : = \bar{D}_{0}  \epsilon^{2} \left( \frac{4n}{n - 1} \max_{\xi \in [0, 1]} \left( \lVert \nabla_{\sigma_{\xi}^{*}  \bar{g}} \bar{u} \rVert_{\calL^{2}(W_{\xi}, \sigma_{\xi}^{*}\bar{g})}^{2} \right) + \max_{W} \lvert R_{\bar{g}} \rvert \lVert u \rVert_{\calL^{2}(W, \sigma^{*}\bar{g})}^{2} \right)^{\frac{1}{2}} \\
& \qquad + \epsilon^{2 + \frac{n}{2}} D_{0}D_{2}D_{s}  \lVert F \rVert_{\calL^{2}(W, \sigma^{*}(\epsilon^{-2}\bar{g}))} \\
& < \frac{8n}{n - 1} \bar{D}_{0} D_{5} \epsilon^{2} \eta + D_{0}D_{2}D_{s} \epsilon^{2} \text{Vol}_{g}(X) (C + 1) \epsilon^{\frac{1}{2}} \\
& : = \bar{D} \epsilon^{2} \eta + \bar{D}' \epsilon^{2 + \frac{1}{2}} (C + 1).
\end{align*}
It follows by the smallness of $ \eta $ in Proposition \ref{Set:prop1} that
\begin{equation*}
\lVert \frac{\partial^{2}u}{\partial t^{2}} \rVert_{\calC^{0}(U_{1})} = \epsilon^{-2} \lVert \frac{\partial^{2}u}{\partial (t')^{2}} \rVert_{\calC^{0}(O_{1})} < \bar{D} \eta + \bar{D}'(C + 1) \epsilon^{\frac{1}{2}} : = \eta' \ll 1.
\end{equation*}
As desired.
\end{proof}
\begin{remark}
    Since $ u_{0} = u + 1 $, Lemma \ref{Set:lemma3} implies that
    \begin{equation*}
        \left\lVert \frac{\partial^{2} u_{0}}{\partial t^{2}} \right\rVert_{\calC^{0}(X \times \lbrace 0 \rbrace_{\xi} \times \lbrace P \rbrace_{t})} < \eta' \Rightarrow \left\lvert \frac{\partial^{2} u'}{\partial t^{2}} \bigg|_{X \times \lbrace 0 \rbrace_{\xi} \times \lbrace P \rbrace_{t}} \right\rvert < \eta'.
    \end{equation*}
\end{remark}

We close this section by comparing $ \Delta_{\bar{g}} u' $, $ \Delta_{\sigma^{*}\bar{g}} u_{0} $ and $ \Delta_{\tau_{1}^{*}\bar{g}} \tilde{u}' $ on $ X \times \lbrace 0 \rbrace_{\xi} \times \lbrace P \rbrace_{t} \subset W \times \lbrace 0 \rbrace_{\xi} \subset M \times \mathbb{S}^{1}_{t} $.
\begin{lemma}\label{Set:lemma4}
Under the hypotheses of Lemma \ref{Set:lemma3},
\begin{equation}\label{Set:eqn8}
\begin{split}
\Delta_{\bar{g}} u' = \Delta_{\tau_{1}^{*}\bar{g}} \tilde{u}' + \frac{\partial^{2} u'}{\partial t^{2}} \; {\rm on} \; X \times \lbrace 0 \rbrace_{\xi} \times \lbrace P \rbrace_{t}, \\
\Delta_{\bar{g}}u' = \Delta_{\sigma^{*}\bar{g}} u_{0} + A_{1}(\bar{g}, M, u)
\end{split}
\end{equation}
where $ A_{1}(\bar{g}, M, u) $ involves zeroth and first order derivatives of $ u $, which does not contain $ \partial_{t} $- terms.
\end{lemma}
\begin{proof}
We can check both equations of (\ref{Set:eqn8}) for local coordinates and local frames on charts containing points of $ Q \in X \times \lbrace 0 \rbrace_{\xi} \times \lbrace P \rbrace_{t} $. Note that $ u' = u_{0} \circ \pi $, $ u_{0} = u' \circ \sigma $, and $ \tilde{u}' = u' \circ \tau_{1} $, the rest of the proof follows exactly the same as in \cite[Lemma 2.4]{RX2}
.

\end{proof}
\medskip

\section{Positive Scalar Curvature on $ X $ with Geometric Angle Condition}
In this section, we show that for any metric $ g $ on the compact cylinder $ M $ with $ R_{g} > 0 $ and $ h_{g} \geqslant 0 $ and satisfying the $ g $-angle condition (\ref{intro:eqn1}), there exists a conformal metric $ \tilde{g}' \in [g] $ such that $ R_{\tilde{g}'} > 0 $ and $ R_{\imath^{*} \tilde{g}'} > 0 $ simultaneously. 

We define
\begin{equation*}
    e^{2\phi'} : = (u')^{\frac{4}{n - 2}} \Rightarrow \phi_{0} : = \sigma^{*} \phi', \tilde{\phi}' : = \tau_{1}^{*} \phi' \Rightarrow e^{2\phi_{0}} : = u_{0}^{\frac{4}{n - 2}}, e^{\tilde{\phi}'} = (\tilde{u}')^{\frac{4}{n - 2}}.
\end{equation*}
For any two functions $ \varphi, v $ with the relation $ e^{2\varphi} = v^{\frac{4}{n - 2}}, n \geqslant 3 $ and any Riemannian metric $ g_{0} $, we have
\begin{align*}
    & v^{-\frac{n+2}{n- 2}}\left(-\frac{4}{n - 2} \Delta_{g_{0}} v 
    \right) = e^{-2\varphi} \left(  
    - 2 \Delta_{g_{0}} \varphi - (n - 2) \lvert \nabla_{g_{0}} \varphi \rvert^{2} \right), \\
   &  e^{-2\varphi} \nabla_{g_{0}} \varphi = \frac{2}{n - 2} v^{-\frac{n+2}{n-2}} \nabla_{g_{0}} v,\  e^{-2\varphi} \lvert \nabla_{g_{0}} \varphi \rvert^{2} = \left( \frac{2}{n - 2} \right)^{2}v^{-\frac{2n}{n - 2}} \lvert \nabla_{g_{0}} v \rvert^{2}, \\
   & e^{-2\tilde{\phi}'} \left( 2(n - 2) \nabla_{V} \nabla_{V} \tilde{\phi}' + (n - 2)^{2} \nabla_{V} \tilde{\phi}' \nabla_{V} \tilde{\phi}' \right)   
= 4(\tilde{u}')^{-\frac{n + 2}{n - 2}}   \nabla_{V} \nabla_{V} \tilde{u}'.
\end{align*}  
Our main theorem gives the existence of the PSC metric on $ X \cong X \times \lbrace 0 \rbrace_{\xi} $. We point out that the proof of Theorem \ref{POS:thm1} follows closely the argument of \cite[Theorem 3.1]{RX2}. 
\begin{theorem}\label{POS:thm1}
Let $ X $ be a closed, oriented manifold with $ \dim X \geqslant 2 $, and $ I_{\xi} = [0, 1]_{\xi} $ be the unit interval with $ \xi $-variable. Let $ M = X \times I_{\xi} $ be the compact cylinder. Assume that $ (M, \partial M), \dim M \geqslant 3 $ admits a metric $ g $ with $ R_{g} > 0 $ and $ h_{g} \geqslant 0 $. If $ g $ satisfies the angle condition
    \begin{equation}\label{POS:eqn0}
    \cos \left(\angle_{g}\left(\nu_{g},\partial_{\xi} \right) \right) > \frac{\sqrt{2}}{2} \Leftrightarrow \frac{g(\partial_{\xi}, \partial_{\xi})}{g(\nu_{g}, \partial_{\xi})^{2}} < 2
    \end{equation}
on $ X \times \lbrace 0 \rbrace_{\xi} $, then there exists a metric $ \tilde{g} $ on $ (M, \partial M) $ such that $ \imath^{*} \tilde{g} $ has positive scalar curvature on $ X \times \lbrace 0 \rbrace_{\xi} \cong X $.
\end{theorem}
\begin{proof}
As we discussed before, we may assume that $ R_{g} > 0 $ and $ h_{g} \equiv 0 $ up to a possible conformal transformation. Recall that the $ g $-angle condition is invariant under conformal transformation, hence is unchanged.

We set $ p > n $ as in Proposition \ref{Set:prop1}. Pick $ C > 0 $ such that
\begin{equation}\label{POS:eqn1}
C > 2 \max_{\partial M} \left( \lvert R_{g} \rvert + 2 \lvert {\rm Ric}_{g}(\nu_{g}, \nu_{g}) \rvert + h_{g}^{2} + \lvert A_{g} \rvert^{2} \right)  + 3.
\end{equation}
We choose $ \eta, \eta' \ll 1 $, and associated small enough $ \delta $, $ \epsilon $ and associated $ F $ in Lemma \ref{Set:lemma2}, such that (i) the solution of (\ref{Set:eqn3}) satisfies  (\ref{Set:eqn4}) and (\ref{Set:eqn5}); (ii) $ \lvert A_{1}(\bar{g}, M, u) \rvert < 1 $ on $ X \times \lbrace 0 \rbrace_{\xi} \times \lbrace P \rbrace_{t} $. Note that $ \tau_{1}^{*}\bar{g} = g $, $ \nabla_{\tau_{1}^{*} \bar{g}} $ and $ \nabla_{V} $ do not contain $ \xi $-derivative. It follows that
\begin{equation}\label{POS:eqn1a}
   \frac{1}{2} < \lVert u_{0} \rVert_{\calC^{0}(X \times [-\frac{\epsilon}{2}, \frac{\epsilon}{2}]_{t})} < \frac{3}{2} \Rightarrow \left\lvert \frac{4}{n - 2} \frac{\left\lvert \nabla_{\tau_{1}^{*} \bar{g}} u_{0} \right\rvert^{2}}{u_{0}} + \frac{4n}{n - 2} \frac{\nabla_{V} u_{0} \nabla_{V} u_{0}}{u_{0}} \right\rvert < 1 \; {\rm on} \; X \cong X \times \lbrace 0 \rbrace_{\xi}.
\end{equation}
For the metric $ \tilde{g} = e^{2\phi'} \bar{g} = (u')^{\frac{4}{n - 2}} \bar{g} $ on $ M \times \mathbb{S}^{1}_{t} $, the normal vector field on $ X \times \lbrace 0 \rbrace_{\xi} \times \lbrace P \rbrace_{t} $ becomes $ e^{-\tilde{\phi}'}\nu_{g} $. Applying Gauss-Codazzi equation on $ X \cong X \times \lbrace 0 \rbrace_{\xi} \times \lbrace P \rbrace_{t} $ with respect to the metric $ \tau_{1}^{*}\tilde{g} $ on $ M $,
\begin{equation*}
R_{\imath^{*} \tau_{1}^{*} \tilde{g}}
 = R_{\tau_{1}^{*} \tilde{g}} - 2{\rm Ric}_{\tau_{1}^{*} \tilde{g}}\left(e^{-\tilde{\phi}'}\nu_{g}, e^{-\tilde{\phi}'}\nu_{g} \right) + h_{\tau_{1}^{*} \tilde{g}} - \lvert A_{\tau_{1}^{*} \tilde{g}} \rvert^{2}.
\end{equation*}
The conformal transformation of Ricci and scalar curvatures are given by
\begin{align*}
    {\rm Ric}_{\tau_{1}^{*} \tilde{g}}\left(e^{-\tilde{\phi}'}\nu_{g}, e^{-\tilde{\phi}'}\nu_{g}\right) & = e^{-2\tilde{\phi}'}\left({\rm Ric}_{\tau_{1}^{*}\bar{g}}\left(\nu_{g}, \nu_{g} \right) - (n - 2)(\nabla_{\nu_{g}} \nabla_{\nu_{g}} \tilde{\phi}' - \nabla_{\nu_{g}} \tilde{\phi}' \nabla_{\nu_{g}} \tilde{\phi}')\right) \\
    & \qquad - e^{-2\tilde{\phi}'}\left(\Delta_{\tau_{1}^{*}\bar{g}} \tilde{\phi}' + (n - 2) \left\lvert \nabla_{\tau_{1}^{*} \bar{g}} \tilde{\phi}' \right\rvert^{2}\right)\bar{g}_{\xi\xi},\\
  R_{\tau_{1}^{*} \tilde{g}}  &=  e^{-2\tilde{\phi}'} \left(R_{\tau_{1}^{*}\bar{g}} -2(n-1)\Delta_{\tau_{1}^{*}\bar{g}}\tilde{\phi}'
  -(n-2)(n-1)\left\lvert\nabla_{\tau_{1}^{*}\bar{g}}\tilde{\phi}'\right\rvert^2 \right).
  \end{align*}
The conformal transformation of the second fundamental form and mean curvature are given by
\begin{align*}
   \lvert A_{\tau_{1}^{*} \tilde{g}} \rvert^{2} & = e^{-2\tilde{\phi}'} \left( \lvert A_{\tau_{1}^{*} \bar{g}} \rvert^{2} + 
   2 n h_{\tau_{1}^{*}\bar{g}} \frac{\partial \tilde{\phi}'}{\partial \nu_{g}} +  n^2
   \left(\frac{\partial \tilde{\phi}'}{\partial \nu_{g}} \right)^{2}  \right),\\
    h_{\tau_{1}^{*} \tilde{g}}^{2} & = e^{-2\tilde{\phi}'} \left( h_{\tau_{1}^{*} \bar{g}}^{2} +
     2n h_{\tau_{1}^{*}\bar{g}} \frac{\partial \tilde{\phi}}{\partial \nu_{g}}  + 
       n^2\left(\frac{\partial \tilde{\phi}'}{\partial \nu_{g}} \right)^{2} \right).
\end{align*}
By (\ref{Set:eqn6}), $ \nabla_{\nu_{g}} \nabla_{\nu_{g}} \tilde{\phi}' = \nabla_{V} \nabla_{V} \tilde{\phi}' $. According to the laws of conformal transformations listed above, and the differential relations before Theorem \ref{POS:thm1}, the Gauss-Codazzi equation converts to
\begin{equation}\label{POS:eqn2}
    \begin{split}
    R_{\imath^{*} \tau_{1}^{*} \tilde{g}} & = e^{-2\tilde{\phi}'} \left(R_{\tau_{1}^{*}\bar{g}} - 2{\rm Ric}_{\tau_{1}^{*}\bar{g}}\left(\nu_{g}, \nu_{g} \right) + h_{\tau_{1}^{*} \bar{g}}^{2} - \lvert A_{\tau_{1}^{*} \bar{g}} \rvert^{2} \right) \\
    & \qquad + e^{-2\tilde{\phi}'} \left(- 2(n - 2) \Delta_{\tau_{1}^{*}\bar{g}} \tilde{\phi}' - (n - 3)(n - 2) \left\lvert \nabla_{\tau_{1}^{*}\bar{g}} \tilde{\phi}' \right\rvert \right) \\
    & \qquad \qquad + 2(n - 2)e^{-2\tilde{\phi}'}\left(\nabla_{V} \nabla_{V} \tilde{\phi}' -  \nabla_{V} \tilde{\phi}' \nabla_{V} \tilde{\phi}' \right) \\
    & = (\tilde{u}')^{-\frac{n +2}{n - 2}}\left(R_{\tau_{1}^{*}\bar{g}} \tilde{u}' - 2{\rm Ric}_{\tau_{1}^{*}\bar{g}}\left(\nu_{g}, \nu_{g} \right)\tilde{u}' + h_{\tau_{1}^{*} \bar{g}}^{2}\tilde{u}' - \lvert A_{\tau_{1}^{*} \bar{g}} \rvert^{2}\tilde{u}' \right) \\
    & \qquad + (\tilde{u}')^{-\frac{n +2}{n - 2}}\left( 4\nabla_{V} \nabla_{V} \tilde{u}' - 4\Delta_{\tau_{1}^{*}\bar{g}} \tilde{u}' \right) \\
    & \qquad \qquad + (\tilde{u}')^{-\frac{n +2}{n - 2}} \left( \frac{4}{n - 2} \frac{\left\lvert \nabla_{\tau_{1}^{*} \bar{g}} \tilde{u}' \right\rvert^{2}}{\tilde{u}'} - \frac{4n}{n - 2} \frac{\nabla_{V} \tilde{u}' \nabla_{V} \tilde{u}'}{\tilde{u}'} \right).
    \end{split}
\end{equation}
Note that on $ X \cong X \times \lbrace 0 \rbrace_{\xi} \times \lbrace P \rbrace_{t} \subset W \times \lbrace 0 \rbrace_{\xi} $, $ \tilde{u}' = u_{0} $. Recall that $ u_{0} = u + 1 $, which satisfies the following PDE
\begin{equation*}
    4\nabla_{V} \nabla_{V} u_{0} - 4\Delta_{\sigma^{*} \bar{g}} u_{0} + R_{\bar{g}} |_{W} u_{0} = F + R_{\bar{g}} |_{W} \; {\rm on} \; X.
\end{equation*}
Since $ \bar{g} = g \oplus dt^{2} $, it follows that $ R_{\bar{g}} = R_{\tau_{1}^{*}\bar{g}} = R_{g} > 0 $ on $X \times \lbrace 0 \rbrace_{\xi} $ by hypothesis. Applying (\ref{Set:eqn5}) of Lemma \ref{Set:lemma3}, (\ref{Set:eqn8}) of Lemma \ref{Set:lemma4}, Remark 2.3, and (\ref{POS:eqn1a}), the formula (\ref{POS:eqn2}) on $ X \cong X \times \lbrace 0 \rbrace_{\xi} \times \lbrace P \rbrace_{t} $ becomes
\begin{equation}\label{POS:eqn3}
\begin{split}
 R_{\imath^{*} \tau_{1}^{*} \tilde{g}} & =u_{0}^{-\frac{n +2}{n - 2}}\left(  - 2{\rm Ric}_{\tau_{1}^{*}\bar{g}}\left(\nu_{g}, \nu_{g} \right)u_{0} + h_{\tau_{1}^{*} \bar{g}}^{2} u_{0} - \lvert A_{\tau_{1}^{*} \bar{g}} \rvert^{2} u_{0} \right) \\
    & \qquad + u_{0}^{-\frac{n +2}{n - 2}}\left( 4\nabla_{V} \nabla_{V} u_{0} - 4\Delta_{\tau_{1}^{*}\bar{g}} \tilde{u}' - 4\frac{\partial^{2} \tilde{u}'}{\partial t^{2}} + 4\frac{\partial^{2} \tilde{u}'}{\partial t^{2}} + R_{\tau_{1}^{*} \bar{g}} u_{0} \right) \\
    & \qquad \qquad + u_{0}^{-\frac{n +2}{n - 2}} \left( \frac{4}{n - 2} \frac{\left\lvert \nabla_{\tau_{1}^{*} \bar{g}} u_{0} \right\rvert^{2}}{u_{0}} - \frac{4n}{n - 2} \frac{\nabla_{V} u_{0} \nabla_{V} u_{0}}{u_{0}} \right) \\
    & \geqslant u_{0}^{-\frac{n +2}{n - 2}}\left(- 2{\rm Ric}_{\tau_{1}^{*}\bar{g}}\left(\nu_{g}, \nu_{g} \right)u_{0} + h_{\tau_{1}^{*} \bar{g}}^{2} u_{0} - \lvert A_{\tau_{1}^{*} \bar{g}} \rvert^{2} u_{0} \right) \\
    & \qquad + u_{0}^{-\frac{n +2}{n - 2}}\left( 4\nabla_{V} \nabla_{V} u_{0} - 4\Delta_{\bar{g}} u' - 4\eta' + R_{\bar{g}} |_{W} u_{0}\right) + u_{0}^{-\frac{n +2}{n - 2}} \cdot (-1) \\
    & = u_{0}^{-\frac{n +2}{n - 2}}\left( - 2{\rm Ric}_{\tau_{1}^{*}\bar{g}}\left(\nu_{g}, \nu_{g} \right)u_{0} + h_{\tau_{1}^{*} \bar{g}}^{2} u_{0} - \lvert A_{\tau_{1}^{*} \bar{g}} \rvert^{2} u_{0} \right) \\
    & \qquad + u_{0}^{-\frac{n +2}{n - 2}}\left( 4\nabla_{V} \nabla_{V} u_{0} - 4\Delta_{\sigma^{*}\bar{g}} u_{0}  + R_{\bar{g}} |_{W} u_{0}\right) + u_{0}^{-\frac{n +2}{n - 2}} \cdot (-1 - 4 \eta' + A_{1}(\bar{g}, M, u)) \\
    & = u_{0}^{-\frac{n +2}{n - 2}}\left(- 2{\rm Ric}_{\tau_{1}^{*}\bar{g}}\left(\nu_{g}, \nu_{g} \right)u_{0} + h_{\tau_{1}^{*} \bar{g}}^{2} u_{0} - \lvert A_{\tau_{1}^{*} \bar{g}} \rvert^{2} u_{0} \right) \\
    & \qquad u_{0}^{-\frac{n +2}{n - 2}} \left(F + R_{\bar{g}} |_{W} -1 - 4 \eta' + A_{1}(\bar{g}, M, u) \right).
\end{split}
\end{equation}
By Lemma \ref{Set:lemma2}, $ F = C + 1 $ on $ \partial M $. By (\ref{POS:eqn1}), (\ref{POS:eqn3}) implies that
\begin{equation*}
    R_{\imath^{*} \tau_{1}^{*} \tilde{g}} \geqslant  u_{0}^{-\frac{n +2}{n - 2}} \left( C + 1 + R_{\bar{g}} - 2 \max_{\partial M} \left(2 \lvert {\rm Ric}_{g}(\nu_{g}, \nu_{g}) \rvert + h_{g}^{2} + \lvert A_{g} \rvert^{2} \right) - 3 \right) > 0. 
\end{equation*}
Therefore, the metric
\begin{equation*}
 \tau_{1}^{*} \tilde{g} = \tau_{1}^{*}(u')^{\frac{4}{n- 2}} \bar{g} = (\tilde{u}')^{\frac{4}{n- 2}} g : = g', g' \in [g]
\end{equation*}
induces a PSC metric $ \imath^{*} g' \in [\imath^{*} g] $ on $ X \cong X \times \lbrace 0 \rbrace_{\xi} $.
\end{proof}
\begin{remark}\label{POS:re1}
    Set $ X_{\xi} = X \times \lbrace \xi \rbrace, \xi \in [0, 1] $, we can define the second fundamental form, unit normal vector field, etc. on the hypersurface $ X_{\xi} $ also. We can then apply our $ g $-angle condition on $ X_{\xi} $. It follows that if $ M $ admits a metric $ g $ with $ R_{g} > 0 $ and $ h_{g} \geqslant 0 $, and the $ g $-angle condition on $ X_{\xi} $ holds, then there exists a metric $ \tilde{g} $ such that $ \imath_{\xi}^{*} \tilde{g} $ has positive scalar curvature on $ X_{\xi} $. Here $ \imath_{\xi} : X_{\xi} \rightarrow M $ is the natural inclusion.
\end{remark}

With Remark \ref{POS:re1}, we generalize the scalar and mean curvature comparison result of Gromov and Lawson \cite{GL} for $ X = \mathbb{T}^{n}, n \geqslant 1 $, and the result of R\"ade for $ \dim X \leqslant 7 $, $ \dim X \neq 4 $.
\begin{corollary}\label{POS:cor2}
     Let $ X $ be a closed, oriented manifold with $ \dim X \geqslant 2 $. If $ X $ admits no PSC metric, meanwhile $ X \times I $ admits a PSC metric $ g $. If the $ g $-angle condition holds on some hypersurface $ X_{\xi}, \xi \in [0, 1] $, then the mean curvature $ h_{g} $ must be negative somewhere.
\end{corollary}

We now give another geometric consequence of Theorem \ref{POS:thm1}. Theorem \ref{POS:thm1} says that for PSC metric $ g $ on $ M $, the Yamabe constant for the conformal class $ [\imath^{*} g] $ on $ X \cong X \times \lbrace 0 \rbrace_{\xi} $ is positive. The next result shows that there exists a metric in the conformal class $ [g] $ which admits PSC metric on $ M $ and induced PSC metric on $ X $.
\begin{corollary}\label{POS:cor1}
Let $(M, \partial M, g) $ be the same as in Theorem \ref{POS:thm1}, $ \dim M = n $. 
If
\begin{equation*}
\cos \left(\angle_g(\partial_{\xi}, \nu_{g}) \right) > \frac{\sqrt{2}}{2}
\end{equation*}
on $ X \times \lbrace 0 \rbrace_{\xi} $, then there exists a conformal metric $ \tilde{g} \in [g] $ such that $ R_{\tilde{g}} > 0 $ and $ R_{\imath^{*} \tilde{g}} > 0 $.
\end{corollary}
\begin{proof}
Let $ \eta_{1, g} $ be the first eigenvalue of the conformal Laplacian $ -\frac{4(n - 1)}{n- 2} \Delta_{g} + R_{g} $ on $ M $ with Neumann boundary condition. It is well-known that $ R_{g} > 0 $ and $ h_{g} \geqslant 0 $ implies that $ \eta_{1, g} > 0 $.

Let $ \zeta_{1, \imath^{*}g} $ be the first eigenvalue of the conformal Laplacian $ -\frac{4(n - 2)}{n - 3} \Delta_{\imath^{*}g} + R_{\imath^{*}g} $. By Theorem \ref{POS:thm1}, $ R_{\imath^{*} g_{1}} > 0 $ for some $ g_{1} \in [g] $. It follows that $ \zeta_{1, \imath^{*}g_{1}} > 0 $, see e.g. \cite{PL}. By \cite[Theorem 3.2]{KW}, $ \zeta_{1, \imath^{*}g} > 0 $ as the sign is preserved under conformal change.

Since $ \eta_{1, g} > 0 $, \cite[Theorem 1.1]{XU7} implies that there exists a metric $ g' \in [g] $ such that $ R_{\tilde{g}} $ is a positive constant, and $ R_{\imath^{*} g'} $ is also a constant.  Therefore $ R_{\imath^{*} g'} > 0 $ since otherwise it contradicts the positivity of $ \zeta_{1, \imath^{*} g'} $.
\end{proof}
\medskip

\section{Positive Scalar Curvature on $ X $ when $ \chi(X) = 0 $ and $ \dim X \geq 5 $}
In this section, we show that the angle condition (\ref{intro:eqn1}) will hold along some separable hypersurface of $ X \times I $, provided that the Euler characteristic $ \chi(X) = 0 $. By a topological result and a slight variation of Theorem \ref{POS:thm1}), we will show that $ X $ admits a PSC metric when $ \chi(X) = 0 $ and $ \dim X \geq 5 $ when $ X \times I $ admits a PSC metric that is mean convex along both boundary components. These results significantly generalize the results given in \S3. In particular, they all hold when $ \dim X = 2m - 1, m \geq 3 $. Our argument of showing that $ \chi(X) = 0 $ results in the satisfaction of (\ref{intro:eqn1}) is inspired by our work in \cite{XU12}.

Let $ \pi_{I} : X \times I \rightarrow I $ be the natural projection to the $ I $-direction. The global variable $ \xi $ is actually defined as
\begin{equation*}
    \xi = \pi_{I}(p), \forall p \in X \times I.
\end{equation*}
We first introduce an equivalent interpretation of the angle condition (\ref{intro:eqn1}), which will be more consistent with the argument in this section.

We claim that on $ X \times \{ 0 \} = \{ \xi = 0 \}$,
\begin{equation}\label{Euler:eqn1}
\cos^{-1} \left( \frac{g(\partial_{\xi}, \nu_{g})}{g(\partial_{\xi}, \partial_{\xi})^{\frac{1}{2}}} \right) \in \left[0, \frac{\pi}{4}\right) \Leftrightarrow \frac{g(\partial_{\xi}, \nu_{g})}{g(\partial_{\xi}, \partial_{\xi})^{\frac{1}{2}}} \in \left(\frac{\sqrt{2}}{2}, 1\right] \Leftrightarrow g_{\xi \xi} g^{\xi \xi} \in [1, 2).
\end{equation}
For any choice of local coordinates $ (x^{1}, \dotso, x^{n}, \xi) $ around some point $ y \in \{ \xi = 0 \} $, the unit normal vector field can be expressed as
\begin{equation*}
    \nu_{g} = \frac{\text{grad}_{g}\xi}{\sqrt{g(\text{grad}_{g}\xi, \text{grad}_{g}\xi)}} =  \frac{1}{\sqrt{g^{\xi \xi}}} \left( \sum_{j= 1}^{n} g^{\xi j} \frac{\partial}{\partial x^{j}} + g^{\xi \xi} \frac{\partial}{\partial \xi} \right).
\end{equation*}
It follows that
\begin{align*}
   \frac{g(\partial_{\xi}, \nu_{g})}{g(\partial_{\xi}, \partial_{\xi})^{\frac{1}{2}}} & =  \frac{g\left(\partial_{\xi}, \frac{1}{\sqrt{g^{\xi \xi}}} \left( \sum_{j= 1}^{n} g^{\xi j} \frac{\partial}{\partial x^{j}} + g^{\xi \xi} \frac{\partial}{\partial \xi} \right)\right)}{g(\partial_{\xi}, \partial_{\xi})^{\frac{1}{2}}} = \frac{1}{\sqrt{g_{\xi \xi} g^{\xi \xi}}} \left( \sum_{j = 1}^{n} g_{\xi j}g^{\xi j} + g_{\xi \xi} g^{\xi \xi} \right) \\
   & = \frac{1}{\sqrt{g_{\xi \xi} g^{\xi \xi}}}  \delta_{\xi}^{\xi} = \frac{1}{\sqrt{g_{\xi \xi} g^{\xi \xi}}}.
\end{align*}
where $ g_{\xi\xi} g^{\xi \xi} = g(\partial_{\xi}, \partial_{\xi}) g^{-1}(d\xi, d\xi) $. Hence (\ref{Euler:eqn1}) holds. 
\begin{remark}\label{Euler:re1}
Suppose that there exists a smooth map $ F : X \times I \rightarrow X \times I $ that preserves $ X \times \{0\} $ and $ X \times \{1\} $, such that the map
\begin{equation*}
    \tilde{\xi} : = \pi_{I} \circ F(p), p \in X \times I
\end{equation*}
implies a new global coordinate of the $ I $-direction and a new global vector field $ \partial_{\tilde{\xi}} $. We can then define the angle condition along the new hypersurface $ S: = \{ \tilde{\xi} = 1/2 \} $ if it {\it{separates}} \cite[Definition 2.1]{Rade} $ X \times \{ 0 \} $ and $ X \times \{ 1 \} $ in the sense that no connected component of $ (X \times I) \backslash \ S $ admits a path $ \gamma : [0, 1] \rightarrow X \times I $ such that $ \gamma(0) \in X \times \{ 0 \} $ and $ \gamma(1) \in X \times \{ 1 \} $. Set $ \nu_{g, \tilde{\xi}} $ to be the unit normal vector field along $ S $ such that $ \angle_{g}(\nu_{g, \tilde{\xi}}, \partial_{\tilde{\xi}}) \in \left[0, \frac{\pi}{2} \right] $ , we want to show that
\begin{equation}\label{Euler:eqn2}
\cos^{-1} \left( \frac{g(\partial_{\tilde{\xi}}, \nu_{g, \tilde{\xi}})}{g(\partial_{\tilde{\xi}}, \partial_{\tilde{\xi}})^{\frac{1}{2}}} \right) \in \left[0, \frac{\pi}{4} \right) \Leftrightarrow \frac{g(\partial_{\tilde{\xi}}, \nu_{g, \tilde{\xi}})}{g(\partial_{\tilde{\xi}}, \partial_{\tilde{\xi}})^{\frac{1}{2}}} \in \left(\frac{\sqrt{2}}{2}, 1\right] \Leftrightarrow g_{\tilde{\xi} \tilde{\xi}} g^{\tilde{\xi} \tilde{\xi}} \in [1, 2).
\end{equation}
\end{remark}

Analogous to \cite[Theorem 2.1]{XU12}, we will show that (\ref{Euler:eqn2}) holds on $ S $, provided that $ \chi(X) = 0 $. First of all, we need the following technical lemma.
\begin{lemma}\label{Euler:lemma1}
Let $ \partial_{\xi} \in \Gamma(T(X \times I)) $ be the global vector field along $ I $-direction. For any global $ g $-unit vector field $ \Gamma(TX) \ni W $ identified as $ (W, 0) \in \Gamma(TX \oplus TI) $, 
\begin{equation}\label{Euler:eqn3}
    \max_{W \in \Gamma(TX), \lVert W \rVert_{g} = 1} g^{2}\left(\frac{\partial_{\xi}}{\lVert \partial_{\xi} \rVert_{g}}, W\right) \leqslant 1 - \frac{1}{g_{\xi\xi}g^{\xi \xi}}.
\end{equation}
\end{lemma}
\begin{proof}
Denote by $ Y = \frac{\partial_{\xi}}{\lVert \partial_{\xi} \rVert_{g}} $ such that $ \lVert Y \rVert_{g} = 1 $. For any $ p \in X \times I $, $ T_{p}(X \times I) = TX \oplus (TX)^{\perp} $ where $ (TX)^{\perp} $ is the $ g $-normal bundle. It follows that
\begin{equation*}
    Y = Y_{TX} + Y^{\perp}
\end{equation*}
uniquely. By $ g $-orthogonality,
\begin{equation*}
    g^{2}(Y, W) = g^{2}(Y_{TX}, W) \leqslant \lVert Y_{TX} \rVert_{g}^{2} = 1 - \lVert Y^{\perp} \rVert_{g}^{2}, \forall W \in \Gamma(TX)
\end{equation*}
with equality if and only if $ W = \pm\frac{Y_{TX}}{\lVert Y_{TX} \rVert_{g}} $. By flowout theorem, we can choose local coordinates $ (x^{1}, \dotso, x^{n}, \xi) $ within some small chart around $ p $ such that $ T_{x}X = Span \{\partial_{1}, \dotso, \partial_{n}) $ for $ x $ near $ p $; in addition, the last coordinate is the original $ \xi $-coordinate with the original $ \partial_{\xi} $ coordinate vector field locally. In these coordinates, there exist smooth functions $a^i$ defined near $ a $ with
\begin{equation*}
    Y^{\perp} = Y - Y_{TX} = \frac{\partial_{\xi}}{\sqrt{g_{\xi \xi}}} - \sum_{i = 1}^{n} a_{i} \partial_{i}.
\end{equation*}
We now compute both sides of
\begin{equation}\label{Euler:eqn4}
d\xi(Y^{\perp}) = g((d\xi)^{\sharp}, Y^{\perp}).
\end{equation}
On one hand,
\begin{equation*}
  d\xi(Y^{\perp}) = d\xi\left(\frac{\partial_{\xi}}{\sqrt{g_{\xi \xi}}} - \sum_{i = 1}^{n} a_{i} \partial_{i} \right) = \frac{1}{\sqrt{g_{\xi \xi}}}.
\end{equation*}
On the other hand, locally $ (d\xi)^{\sharp} = \sum_{i = 1}^{n} g^{i\xi} \partial_{i} + g^{\xi \xi} \partial_{\xi} $, therefore
\begin{equation*}
    g((d\xi)^{\sharp}, Y^{\perp}) = g\left(\sum_{i = 1}^{n} g^{i\xi} \partial_{i} + g^{\xi \xi} \partial_{\xi}, Y^{\perp}\right) = g^{\xi \xi} g(\partial_{\xi}, Y^{\perp}).
\end{equation*}
By (\ref{Euler:eqn4}), we get
\begin{equation}\label{Euler:eqn5}
d\xi(Y^{\perp}) = g((d\xi)^{\sharp}, Y^{\perp}) \Leftrightarrow \frac{1}{\sqrt{g_{\xi \xi}}} = g^{\xi \xi} g(\partial_{\xi}, Y^{\perp}) \Leftrightarrow g(\partial_{\xi}, Y^{\perp}) = \frac{1}{\sqrt{g_{\xi \xi}} g^{\xi \xi}}.
\end{equation}
Finally by (\ref{Euler:eqn5}), we have
\begin{align*}
    \lVert Y^{\perp} \rVert_{g}^{2} & = g(Y^{\perp}, Y^{\perp}) = g\left(\frac{\partial_{\xi}}{\sqrt{g_{\xi \xi}}} - \sum_{i = 1}^{n} a_{i} \partial_{i}, Y^{\perp} \right) = \frac{1}{\sqrt{g_{\xi \xi}}} g(\partial_{\xi}, Y^{\perp}) = \frac{1}{g_{\xi \xi}g^{\xi \xi}}.
\end{align*}
Hence (\ref{Euler:eqn3}) follows. 
\end{proof}
We are now ready to introduce our main technical result. The key step is to construct a global injective smooth map $ F $ to define a new hypersurface $ \{ \tilde{\xi} = 0 \} $ on which (\ref{Euler:eqn2}) holds. The construction of $ F $ is coming from the global flow of some nowhere vanishing global vector field on $ X \times I $, where the topological condition $ \chi(X) = 0 $ is used. The proof of the following theorem is essentially due to \cite[Theorem 2.1]{XU12}.
\begin{proposition}\label{Euler:prop1}
Let $ X $ be a connected, closed, oriented manifold, $ \dim X \geqslant 2 $. Let $ g $ be a Riemannian metric on $ X \times I $, smooth up to the boundary. If $ \chi(X) $ = 0, then there exists a smooth local diffeomorphism
\begin{equation*}
F : X \times I \rightarrow X \times I
\end{equation*}
such that $ \tilde{\xi} : = \pi_{I} \circ F $ is a global coordinate along $ I $-direction, which follows that
\begin{equation}\label{Euler:eqn6}
 1 < g\left(\partial_{\tilde{\xi}}, \partial_{\tilde{\xi}}\right)g^{-1}\left(d\tilde{\xi}, d\tilde{\xi}\right) < 2 \; {\rm on} \; \left\{ \tilde{\xi} = \frac{1}{2} \right\} = S.
\end{equation}
In addition, $ S $ separates two disjoint boundary components of $ X \times I $ in the sense of \cite[Definition 2.1]{Rade}. 
\end{proposition}
\begin{proof}
Since $ X \times I $ is compact, there exist $ C_{1}, C_{2} > 0 $ such that both $ g_{\xi \xi} \geqslant C_{1} $ and $ g^{\xi \xi} \leqslant C_{2} $ hold uniformly on $ X \times I $. Without loss of generality, we may thus assume that 
\begin{equation*}
    g_{\xi \xi} > 1, \sqrt{g^{\xi \xi}} \leqslant 1 - \zeta \; \text{uniformly on} \; X \times I
\end{equation*} for some $ \zeta \in (0, 1) $ by a one-time scaling of the original metric $ g $ if necessary. 

Since $ \chi(X) = 0 $, there exists a nowhere vanishing global vector field $ \bar{V} \in \Gamma(TX) $. Recall that $ \pi_{X} : X \times I \rightarrow X $ and $ \pi_{I} : X \times I \rightarrow I $ are natural projections. We define
\begin{equation*}
    V \in \Gamma(T(X \times I)), T_{p}(X \times I) \ni V_{p} : = (\bar{V}_{p}, 0) \in T_{\pi_{X}(p)}X \oplus T_{\pi_{I}(p)}I, \forall p \in X \times I.
\end{equation*}
as a global vector field on $ X \times I $. $ V $ is a complete, global, nowhere vanishing vector field whose integral curve starting at any $ p \in X \times I $ is in the fiber $ \{ \xi = \pi_{I}(p) \} $ with infinitesimal generator $ \bar{V} $. Clearly, the global nowhere vanishing vector field $ \partial_{\xi} \in \Gamma(T(X \times I)) $ and $ V $ are linearly independent. Set
\begin{equation*}
U = \xi(1 - \xi)\cdot \frac{V -g\left(\frac{\partial_{\xi}}{\lVert \partial_{\xi} \rVert_{g}}, V \right) \frac{\partial_{\xi}}{\lVert \partial_{\xi} \rVert_{g}}}{\left\lVert V - g\left(\frac{\partial_{\xi}}{\lVert \partial_{\xi} \rVert_{g}}, V \right) \frac{\partial_{\xi}}{\lVert \partial_{\xi} \rVert_{g}} \right\rVert_{g}}
\end{equation*}
It follows from linear independence of $ V $ and $ \partial_{\xi} $ that $ U \in \Gamma(T(X \times I)) $ is a global nowhere vanishing unit vector field in the interior $ X \times (-1, 1) $ and vanishes along $ X \times \{ 0 \} $ and $ X \times \{ 1 \} $. It follows that $ U $ is complete. Clearly, $ g(U, U) = \xi^{2}(1 - \xi)^{2}, g(U, \partial_{\xi}) = 0 $. In addition, $ U $ and $ \partial_{\xi} $ are linearly independent. The flow of $ U $ trivially preserves $ X \times \{ 0 \} $ and $ X \times \{ 1 \} $, respectively.

We denote by $ \Phi : \R \times (X \times I) \rightarrow X \times I $ the global smooth flow of $ U $. We define a map
\begin{equation}\label{Euler:eqn7}
    F : X \times I \rightarrow X \times I, F(p) = \Phi_{t_{p}}(p), \forall p \in X \times I
\end{equation}
such that for each $ p \in X \times I $, $ t = t_{p} $ satisfies, 
\begin{equation}\label{Euler:eqn8}
 \exists t -1 + \pi_{I}(\Phi_{t}(p)) = 0, \forall p \in X \times I \Leftrightarrow \Phi_{t - 1 + \pi_{I}(\Phi_{t}(p))}(p) = p.
\end{equation}
Equivalently, we claim that $ t = t_{q} $ in (\ref{Euler:eqn8}) has a solution for each $q$ and is a smooth function of $ q $. We do not say that such a solution is unique, since it is possible that if $ g(\partial_{\xi}, V) = 0 $ at some point $ p $ so that the flow of $ U $ degenerates to the flow of $ V $, and that the flow of $ V $ admits periodic orbit starts at $ p $. 

To verify this, we {\it{claim}} that for each fixed $ q \in X \times I $
\begin{equation*}
    h_{q}(t) : = t -1 + \pi_{I}(\Phi_{t}(q))
\end{equation*}
is a monotonic function from $ \R $ to $ \R $, and that $ h_{q}(1)  \geq 0 $, and $ h_{q}(0) \leq 0 $. The claim implies that $ h_{q}(t) = 0 $ for a unique $ t \in [0,1] $, and the smoothness in $q$ will follow.

By definition, $ h_{q}(t) $ is smooth in both $ t $ and $ q $. By Lemma \ref{Euler:lemma1}, we compute that
\begin{align}\label{Euler:eqn9}
    \frac{dh_{q}(t)}{dt} & = 1 - \frac{\xi(1 - \xi)g(\frac{\partial_{\xi}}{\lVert \partial_{\xi} \rVert_{g}}, V)}{\lVert \partial_{\xi} \rVert_{g}\left\lVert V -g\left(\frac{\partial_{\xi}}{\lVert \partial_{\xi} \rVert_{g}}, V \right) \frac{\partial_{\xi}}{\lVert \partial_{\xi} \rVert_{g}} \right\rVert_{g}} \geqslant 1 - \frac{\sqrt{g(V, V)}}{\sqrt{g_{\xi \xi}} \sqrt{g(V, V) - \left( g\left(\frac{\partial_{\xi}}{\lVert \partial_{\xi} \rVert_{g}}, V \right)\right)^{2}}} \nonumber \\
    & = \frac{\sqrt{g(V, V)}}{\sqrt{g_{\xi \xi}} \sqrt{g(V, V) - g(V, V)\left( g\left(\frac{\partial_{\xi}}{\lVert \partial_{\xi} \rVert_{g}}, \frac{V}{\lVert V \rVert_{g}} \right)\right)^{2}}} \nonumber \geqslant 1 - \frac{1}{\sqrt{g_{\xi \xi} \cdot \left( 1 - \left(1 - \frac{1}{g_{\xi\xi}g^{\xi\xi}} \right) \right)}} \nonumber \\
     & = 1 - \sqrt{g^{\xi \xi}} \geqslant \zeta > 0.
\end{align}
The first inequality in the last line holds since $ \xi(1 - \xi) \in [0, 1/4] $, and $ g_{\xi \xi}g^{\xi \xi} \geq 1 $. Thus $ h_{q}(t) $ is a monotone increasing function of $ t $ for every $ q \in X \times I $. Since $ \pi_{I}(q) \in [0, 1], \forall q \in X \times I $, it follows that
\begin{equation*}
    h_{q}(0) = 0 - 1 + \pi_{I}(\Phi_{0}(q)) \leqslant - 1 + 1 = 0,
    h_{q}(1) = 1 - 1 + \pi_{I}(\Phi_{1}(q)) \geqslant 0.
\end{equation*}
Therefore, the {\it{claim}} holds. Consequently, (\ref{Euler:eqn8}) holds. In addition, $ t $ depends smoothly on $ q $ by implicit function theorem since $ h_{q}'(t) > 0 $.

Define
\begin{equation*}
    \tilde{\xi} : = \pi_{I}' : = \pi_{I} \circ F.
\end{equation*}
By (\ref{Euler:eqn8}), we treat $ t $ as a smooth map from $ X \times I $ to $ \R $ also. It is straightforward to compute that
\begin{equation*}
    -d\tilde{\xi} = dt = \frac{(\Phi_{t_{q}})^{*}d\pi_{I}}{1 - \frac{\xi(1 - \xi)g(\frac{\partial_{\xi}}{\lVert \partial_{\xi} \rVert_{g}}, V)}{\lVert \partial_{\xi} \rVert_{g}\left\lVert V -g\left(\frac{\partial_{\xi}}{\lVert \partial_{\xi} \rVert_{g}}, V \right) \frac{\partial_{\xi}}{\lVert \partial_{\xi} \rVert_{g}} \right\rVert_{g}}}.
\end{equation*}
By (\ref{Euler:eqn9}), the denominator above is never zero. Therefore $ d\tilde{\xi} \neq 0 $ on $ X \times I $. In addition, $ \tilde{\xi}(p) = 1 $ if $ p \in X \times \{ 1 \} $, and $ \tilde{\xi}(p) = 0 $ if $ p \in X \times \{ 0 \} $. It follows that $ \tilde{\xi} $ is a global coordinate along $ I $-direction. In terms of $ \tilde{\xi} $, we can rewrite $ F $ as
\begin{equation}\label{Euler:rev1}
F(p) = \Phi_{t_{p}}(p) = \Phi_{1- \pi_{I}'(p)}(p).
\end{equation}
Our verification of (\ref{Euler:eqn6}) becomes a pointwise argument with appropriate choices of local coordinates. 

Since $ F $ preserves $ X \times \{ 0 \} $ and $ X \times \{ 1 \} $, we see immediately that $ S = \left\{ \tilde{\xi} = \frac{1}{2} \right\} $ is not the boundary component of $ X \times I $. Fix an arbitrary point $ Q \in S \subset X \times I_{\tilde{\xi}} $. By flowout theorem, we can choose local coordinates $ (x^{1}, \dotso, x^{n}, \tilde{\xi}) $ around the given point $ Q $, possibly different from the local coordinates chosen above, such that just at the point $ Q $ centered as $ (0, \dotso, 0, 0) $, we have $ U(Q) = \partial_{1} $; This choice of local coordinates can be chosen such that the last coordinate is the original global coordinate $ \tilde{\xi} $. By (\ref{Euler:rev1}), the local expression of $ F $ in $ (x^{1}, \dotso, x^{n}, \tilde{\xi}) $ is of the form
\begin{equation}\label{Euler:eqn10}
F(x, \tilde{\xi}) = (x^{1} - \tilde{\xi} + 1, x^{2}, \dotso, x^{n}, \tilde{\xi}) + O(\lvert (x, \tilde{\xi}) \rvert^{2}).
\end{equation}
Note that $ \lvert (x, \tilde{\xi}) \rvert \ll 1 $ in this local coordinate expression (\ref{Euler:eqn10}). It follows from (\ref{Euler:eqn10}) that
\begin{equation}\label{Euler:eqn11}
JF_{Q} = \begin{bmatrix} 1 & 0 & \dotso & 0 & -1 \\ 0 & 1 & \dotso & 0 & 0 \\ \vdots & \vdots & \ddots & \vdots & \vdots \\ 0 & 0 & \dotso & 1 & 0 \\ 0 & 0 & \dotso & 0 & 1 \end{bmatrix}, JF_{Q}^{-1} = \begin{bmatrix} 1 & 0 & \dotso & 0 & 1 \\ 0 & 1 & \dotso & 0 & 0 \\ \vdots & \vdots & \ddots & \vdots & \vdots \\ 0 & 0 & \dotso & 1 & 0 \\ 0 & 0 & \dotso & 0 & 1 \end{bmatrix}.
\end{equation}
Analogously, when we choose another collection of local coordinates $ (y^{1}, \dotso, y^{n}, \tilde{\xi}) $ around any other $ Q' \in S $, we have the same expression of (\ref{Euler:eqn11}). It follows that $ F $ is a local diffeomorphism. Since the domain of $ F $ is the whole $ X \times I $, we can define the global vector field
\begin{equation*}
    \partial_{\tilde{\xi}} |_{p} : = (dF^{-1})_{F(p)} \partial_{\xi} |_{F(p)}
\end{equation*}
which is along the $ I_{\tilde{\xi}} $-direction.

Just at $ Q $, we choose a possibly different local coordinates $ (x^{1}, \dotso, x^{n}, \xi) $ around the given point $ Q $ (note that the last coordinate is $ \xi $ here), such that just at the point $ Q $ centered as $ (0, \dotso, 0, 0) $, we have $ U(Q) = \partial_{1} $; in addition, we have $ g_{ij}(Q) = g^{ij}(Q) = \delta_{i}^{j}, g_{i\xi}(Q) = 0, i, j = 2, \dots, n, $. This choice of local coordinates can be chosen such that the last coordinate is the original global coordinate $ \xi $ with the original vector field $ \partial_{\xi} $ locally around $ Q $. we apply the local expression of F in (\ref{Euler:eqn10}) that
\begin{align*}
g(\partial_{\tilde{\xi}}, \partial_{\tilde{\xi}})(Q) g^{-1}(d\tilde{\xi}, d\tilde{\xi})(Q) & = g(\partial_{x^{1}} + \partial_{\xi}, \partial_{x^{1}} + \partial_{\xi}) g^{-1}(d\xi, d\xi) = (g_{11} + 2g_{1\xi} + g_{\xi \xi}) g^{\xi \xi} \\
& = (g_{11} + 2g_{1\xi} + g_{\xi \xi}) \cdot \frac{g_{11}}{g_{11}g_{\xi \xi} - g_{1 \xi}^{2}}.
\end{align*}
With $ U(Q) = \partial_{1} $, and the $ g $-orthogonality between $ U $ and $ \partial_{\xi} $, we have
\begin{equation*}
g(\partial_{\tilde{\xi}}, \partial_{\tilde{\xi}})(Q) g^{-1}(d\tilde{\xi}, d\tilde{\xi})(Q) = (g_{11} + 2g_{1\xi} + g_{\xi \xi}) \cdot \frac{g_{11}}{g_{11}g_{\xi \xi} - g_{1 \xi}^{2}} = \frac{\xi^{2}(1 - \xi)^{2}}{g_{\xi\xi}} + 1.
\end{equation*}
Recall that $ g_{\xi \xi} > 1 $. It follows that (\ref{Euler:eqn6}) holds at $ Q $. Repeating the same argument at each point of $ S $ with associated choices of local coordinates, we conclude that (\ref{Euler:eqn6}) holds on the hypersurface $ S $.

Recall that $ \tilde{\xi}(q) = 0 < \frac{1}{2}, \forall q \in X \times \{ 0 \} $ and $ \tilde{\xi}(q) = 1 > \frac{1}{2}, \forall q \in X \times \{ 1 \} $. Therefore $ S $ separates $ X \times \{ 0 \} $ and $ X \times \{1 \} $.
\end{proof}
The next technical lemma shows that at least one connected component of $ S $ separates $ X \times \{ 0 \} $ and $ X \times \{ 1 \} $.
\begin{lemma}\label{Euler:lemma2}
Let $ S = \left\{ \tilde{\xi} = \frac{1}{2} \right\} \subset X \times I $  be the hypersurface in Theorem \ref{Euler:prop1}. There exists at least one connected component of $ S $, denoted by $ \Sigma $, that separates $ X \times \{ 0 \} $ and $ X \times \{ 1 \} $.
\end{lemma}
\begin{proof}
Clearly
\begin{equation*}
    S =  \sqcup_{i = 1}^{K} \Sigma_{i}
\end{equation*}
where each $ \Sigma_{i} $ is a connected component of $ S $. Pick a path $ \gamma : [0, 1] \rightarrow X \times I $ where $ \gamma(0) \in X \times \{0 \} $, $ \gamma(1) \in X \times \{ 1 \} $, and $ \gamma \pitchfork S $. By the argument in the last paragraph of Theorem \ref{Euler:prop1}, $ \tilde{\xi} \circ \gamma  - \frac{1}{2}: [0, 1] \rightarrow I $ changes sign with simple zeros. It follows that the intersection number
\begin{equation*}
1 = I_{2}(\gamma, S) = \sum_{i = 1}^{K} I_{2}(\gamma, \Sigma_{i}) \Rightarrow \exists \Sigma_{i} \; {\rm s.t.} \; I_{2}(\gamma, \Sigma_{i}) = 1. 
\end{equation*}
Denote by $ \Sigma : = \Sigma_{i} $, it follows from intersection pairing and the homotopy equivalence $ X \times I \cong X $ that
\begin{equation*}
0 \neq [\Sigma] \in H_{n - 1}(M; \mathbb{Z}_{2}) \cong \mathbb{Z}_{2}.
\end{equation*}
By relative K\"unneth theorem and connectedness of $ X \times I $, we have $ H_{1}(M, \partial M; \mathbb{Z}_{2}) \cong \mathbb{Z}_{2} $. Also from long exact sequence, we have
\begin{equation*}
H_{1}(M, \partial M; \mathbb{Z}_{2}) = \text{Ker}(H_{0}(\partial M; \mathbb{Z}_{2}) \hookrightarrow H_{0}(M; \mathbb{Z}_{2})).
\end{equation*}
Hence every path $ \eta : [0, 1] \rightarrow X \times I $ with $ \eta(0) \in X \times \{ 0 \} $ and $ \eta(1) \in X \times \{ 1 \} $ induces a generator $ [\eta] \in H_{1}(M, \partial M; \mathbb{Z}_{2}) $. By Poincar\'e-Lefschetz duality, the mod-2 intersection pairing
\begin{equation*}
H_{n - 1}(M; \mathbb{Z}_{2}) \times H_{1}(M, \partial M; \mathbb{Z}_{2}) \rightarrow \mathbb{Z}_{2}
\end{equation*}
is nondegenerate, which follows that
\begin{equation*}
    I_{2}(\Sigma, \eta) = 1.
\end{equation*}
Consequently, every path of $ \eta $-type intersects $ \Sigma $. Therefore $ \Sigma $ separates $ X \times \{ 0 \} $ and $ X \times \{ 1 \} $. 
\end{proof}
We need one last technical result for our main result of this section. The assumption $ \dim X \geq 5 $ is required.
\begin{proposition}\cite[Proposition 6.4]{Rade}\label{Euler:prop1}
Let $ X $ be a closed, connected, oriented manifold with $ \dim X \geqslant 5 $. Let $ Y = X \times [-a, a]_{\xi} $ for some $ a \in \R^{+} $. Let $ \Sigma $ be a closed, connected, oriented hypersurface separating $ \partial_{a, -} Y $ and $ \partial_{a, +} Y $. If $ \Sigma $ admits a PSC metric, then so does $ X $. 
\end{proposition}
We are now ready to prove our main result.
\begin{theorem}\label{Euler:thm2}
Let $ X $ be a connected, closed, oriented manifold, $ \dim X \geqslant 5 $, with $\chi(X) = 0$. 
Then $ X \times I $ admits a metric $ g $ with $ R_{g} > 0 $ and $ h_{g} \geq 0 $ if and only if $ X $ admits a PSC metric.
\end{theorem}
\begin{proof}
Since $ X $ is oriented, by Lemma \ref{Euler:lemma2} we can assume without loss of generality that $ S = \left\{ \tilde{\xi} = \frac{1}{2} \right\}$ is a connected, closed, oriented hypersurface that separates $ X \times \{ 0 \} $ and $ X \times \{ 1 \} $. By Theorem \ref{Euler:prop1} and (\ref{Euler:eqn6}), (\ref{Euler:eqn2}) holds along $ S $. 

Assume that $ X \times I $ admits $ g $ with $ R_{g} > 0 $ and $ h_{g} \geq 0 $. Applying Theorem \ref{POS:thm1} with angle condition (\ref{Euler:eqn2}) on $ S $, it follows that $ S $ admits a PSC metric. Since $ \dim X \geq 5 $, Proposition \ref{Euler:prop1} implies that $ X $ admits a PSC metric.

The other direction is trivial by constructing a product metric on $ X \times I $.
\end{proof}
We can further generalize the result of Gromov and Lawson \cite{GL} by imposing a topological condition instead of the geometric angle condition.
\begin{corollary}\label{Euler:cor1}
Let $ X $ be a closed, oriented manifold with $ \dim X \geqslant 5 $ and $ \chi(X) = 0 $. If $ X $ admits no PSC metric, meanwhile $ X \times I $ admits a PSC metric $ g $, then the mean curvature $ h_{g} $ must be negative somewhere on $ X \times \lbrace 0 \rbrace \cup X \times \lbrace 1 \rbrace $.
\end{corollary}
\begin{proof}
This is exactly the contrapositive statement of Theorem \ref{Euler:thm2}.
\end{proof}
  
\bibliographystyle{plain}
\bibliography{ScalarPre}
\end{document}